\documentclass[12pt]{amsart}

\usepackage{amsthm,amssymb}
\usepackage{amsmath}
\usepackage[framemethod=tikz]{mdframed}

\usepackage{url}
\usepackage[colorlinks,linkcolor=blue,anchorcolor=blue,citecolor=blue,backref=page]{hyperref}
\usepackage{color}
\usepackage{graphics,epsfig}
\usepackage{graphicx}
\usepackage{float} 
\usepackage[english]{babel}
\usepackage{mathtools}
\usepackage{todonotes}
\usepackage{url}
\usepackage{cite}
\usepackage[colorlinks,linkcolor=blue,anchorcolor=blue,citecolor=blue,backref=page]{hyperref}

\usepackage{bm}

\usepackage[norefs,nocites]{refcheck}

\hypersetup{breaklinks=true}

\newtheorem{thm}{Theorem}
\newtheorem{lem}[thm]{Lemma}

\newtheorem{cor}[thm]{Corollary}

\newtheorem{rem}[thm]{Remark}

\newtheorem{lemma}[thm]{Lemma}

\newcommand{\GL}{\operatorname{GL}}

\newcommand{\Tr}{\operatorname{Tr}}

\numberwithin{equation}{section}
\numberwithin{thm}{section}
\numberwithin{table}{section}

\def\vol {{\mathrm{vol\,}}}
\def\squareforqed{\hbox{\rlap{$\sqcap$}$\sqcup$}}
\def\qed{\ifmmode\squareforqed\else{\unskip\nobreak\hfil
\penalty50\hskip1em\null\nobreak\hfil\squareforqed
\parfillskip=0pt\finalhyphendemerits=0\endgraf}\fi}

\def \balpha{\bm{\alpha}}

\def\supp{\text{supp}}

\def\cB{{\mathcal B}}

\def\cD{{\mathcal D}}

\def\cF{{\mathcal F}}

\def\cI{{\mathcal I}}
\def\cJ{{\mathcal J}}

\def\cL{{\mathcal L}}

\def\cS{{\mathcal S}}

\def\inv{\mathrm{inv}} 
\def\sqrt{\mathrm{sqrt}}

\def \F {{\mathbb F}}

\def \R {{\mathbb R}}
\def \Z {{\mathbb Z}}

\def \Ki {\F_q(T)_\infty}

\def\\{\cr}
\def\({\left(}
\def\){\right)}

\def\Tr{{\mathrm{Tr}}}

\def\Res{\mathrm{Res}}

\newcommand{\ov}{\overbar}
\newcommand{\overbar}[1]{\mkern 1.5mu\overline{\mkern-1.5mu#1\mkern-1.5mu}\mkern 1.5mu}

 \newcommand{\Mod}[1]{\ (\mathrm{mod}\ #1)}

\newcommand{\comm}[1]{{\color{magenta}   Bryce: \hrule #1\hrule}}

\makeatletter
\def\l@subsection{\@tocline{2}{0pt}{2.8pc}{1pc}{}}
\def\l@subsubsection{\@tocline{2}{0pt}{5pc}{7.5pc}{}}
\makeatother

\usepackage{longtable}

 \author[C.~Bagshaw]{Christian Bagshaw}
 \address{School of Mathematics and Statistics, University of New South Wales.
 Sydney}
 \email{c.bagshaw@unsw.edu.au}

 \author[B.~Kerr]{Bryce Kerr}
 \address{School of Science, University of New South Wales.
 Canberra}
 \email{bryce.kerr@unsw.edu.au}

\title{Lattices in Function Fields and Applications  }

\keywords{function field, finite field, geometry of numbers, lattices, convex bodies, polynomial curves, modular inversions, modular roots}
\subjclass[2010]{11T06, 11H06}

\begin{document}

\maketitle

\begin{abstract}
In recent decades, the use of ideas from Minkowski's Geometry of Numbers has gained recognition as a helpful tool in bounding the number of solutions to modular congruences with variables from short intervals. In 1941, Mahler introduced an analogue to the Geometry of Numbers in function fields over finite fields. Here, we build on Mahler's ideas and develop results useful for bounding the sizes of intersections of lattices and convex bodies in $\F_q((1/T))^d$, which are more precise than what is known over $\mathbb{R}^d$. These results are then applied to various problems regarding bounding the number of solutions to congruences in $\F_q[T]$, such as the number of points on polynomial curves in low dimensional subspaces of finite fields. Our results improve on a number of previous bounds due to Bagshaw, Cilleruelo, Shparlinski and Zumalac\'{a}rregui. We also present previous techniques developed by various authors for estimating certain energy/point counts in a unified manner.
\end{abstract}

\tableofcontents

\section{Introduction}
\subsection{Motivation} Minkowski's \textit{Geometry of Numbers} \cite{M1910} introduces a number of notions and results regarding $\Z$-lattices in $\R^d$, and their relationships with symmetric convex bodies. In particular, let $\cL$ be a lattice of the form
$$ \cL = A_\cL \Z^d$$
for some $A_\cL \in \GL_d(\R)$, and let $\cB \subseteq \R^d$ denote an open, convex set which is symmetric about the origin. We will set $\det \cL = |\det A_{\cL}|$. Perhaps the most famous result in this area, often referred to as Minkowski's First Theorem, states that if 
$$\vol \cB > 2^d \det \cL$$
then $\cB$ contains a non-zero point of $\cL$. Minkowski also introduced the idea of \textit{successive minima}. For each $i$ satisfying $1 \leq i \leq d$, we define the $i$-th successive minima of $\cL$ with respect to $\cB$ as 
$$\sigma_i = \inf\{\lambda > 0 : \lambda \cB \text{ contains $i$ linearly independent vectors of $\cL$}\}.$$
 Minkowski's Second Theorem states that
\begin{align}\label{eq:Minkowski_second_theorem}
\frac{2^n}{n!} \frac{\det \cL}{\vol \cB}\le \sigma_1...\sigma_d \leq 2^n \frac{\det \cL}{\vol \cB},
\end{align}
from which one may recover the first theorem as a special case. It is possible to use~\eqref{eq:Minkowski_second_theorem} to count points in the intersection $\cL\cap \cB$ when $\cB$ is `large enough'. But for general $\cB,\cL$  an important result, given in \cite{H2002}, states that
\begin{align}\label{eq:Henk}
    |\cB \cap \cL| < 2^{d-1}\prod_{i=1}^d\left\lfloor \frac{2}{\sigma_i} + 1\right\rfloor.
\end{align}
The above results have been used extensively throughout many areas of mathematics, but in particular in number theory.

In this paper, we focus on applications of the theory of successive minima to problems which aim to determine when we can lift integral solutions $(x_1,\dots,x_n)$ of congruences 
\begin{align}
\label{eq:eqneqneqneqneqneqneqqneqnneqnefnqenqeneq}
f(x_1,\dots,x_n)\equiv 0 \Mod{q}
\end{align}
to integral solutions $(\tilde{x}_1,\dots,\tilde{x}_n)$ of equations 
\begin{align}
\label{eq:11-1-1-1-1-1}
\tilde f(\tilde x_1,\dots,\tilde x_n)= 0
\end{align}
with $f,\tilde f$ certain rational functions. A natural restriction on the variables $x_1,\dots,x_n$ is that they have size bounded by some parameter 
\begin{align}
\label{eq:xjbounds}
|x_j|\le H_j, \quad 1\le j \le n    
\end{align}
 which is a scenario that frequently occurs in number theory. In certain cases it is easier to count solutions to equations~\eqref{eq:11-1-1-1-1-1} than congruences~\eqref{eq:eqneqneqneqneqneqneqqneqnneqnefnqenqeneq} (for example by taking advantage of order preserving properties of addition and multiplication) and this phenomenon motivates lifting the congruence~\eqref{eq:eqneqneqneqneqneqneqqneqnneqnefnqenqeneq} to the equation~\eqref{eq:11-1-1-1-1-1}. 

One early example of this strategy occurs in Heath-Brown's work on the divisor function in arithmetic progressions~\cite{Heath1986} where estimates for the number of solutions to congruences 
\begin{align}
\label{eq:hb-1}
\frac{1}{x_1}+\frac{1}{x_2}\equiv \lambda \Mod{q}, \quad 1\le x_1,x_2\le H,
\end{align}
are required. One may apply Dirichlet's pigeonhole principle to obtain small integers $a,b$ satisfying $\lambda \equiv ab^{-1} \Mod{q}$ which after substituting into the above and clearing denominators, reduces to counting solutions to 
$$(x_1+\alpha_1)(x_2+\alpha_2)=\lambda_1,$$
for some $\alpha_1$ and $\alpha_2$ in terms of $a$ and $b$ and small $\lambda_1 \ll q$, to which one can apply bounds for the divisor function. Working through the details provides a bound of $H^{o(1)}$ for the number of solutions to~\eqref{eq:hb-1} provided $H\le q^{1/3}.$ A significant open problem is to improve on the exponent $1/3$. One expects a bound of the form $H^{o(1)}$ in a longer range $H\le q^{1/2},$ which can be shown to be optimal via a density argument.  

It is possible to interpret Heath-Brown's use of the pigeonhole principle as a special case of Minkowski's first theorem. This allows for a wide-reaching generalisation of the above idea: with variables satisfying~\eqref{eq:xjbounds}, when can ~\eqref{eq:eqneqneqneqneqneqneqqneqnneqnefnqenqeneq} be lifted to~\eqref{eq:11-1-1-1-1-1}? In the case of $f$ a polynomial, after defining a suitable lattice a straightforward application of Minkowski's theorem to establish the existence of a small lattice point shows this is possible for $H$  small enough, depending only on the degree $d$ and number of variables $n$ of $f$. Important developments of this idea are due to Coppersmith (see for example~\cite{Coppersmith1997}) with a focus on applications to cryptography. In the context of number theory, techniques related to the above circle of ideas have been developed by a number of different authors. Some recent examples and applications include:
\begin{enumerate}
    \item A series of paper by Bourgain, Garaev, Konyagin and Shparlinski~\cite{doi:10.1137/110850414,Bourgain2013,Bourgain2014} who deal with multiplicative congruences with various applications to computer science and character sums.
    \item Bourgain and Garaev~\cite{BourgainGaraev2014,J2014} who consider equations with Kloosterman fractions with applications to the Brun-Titchmarsh theorem.
    \item Papers of Chang~\cite{Mei2014} and Cilleruelo, Garaev, Ostafe and Shparlinski~\cite{Cilleruelo2012} on solutions to polynomial congruences with applications to algebraic dynamical systems.
    \item Dunn, Kerr, Shparlinski and Zaharescu~\cite{DKSZ2020} (see also~\cite{KSSZ2021}) who consider equations with modular square roots motivated by applications to moments of half integral weight modular forms.
    \item Kerr, Shparlinski, Wu and Xi~\cite{https://doi.org/10.48550/arxiv.2204.05038} who consider various equations with Kloosterman fractions which have been applied to the error term in the fourth moment of Dirichlet $L$-functions. 
\end{enumerate}

There has been some interest in extending the  results listed above into the setting of function fields over finite fields, such as work of Bagshaw and Shparlinski~\cite{BS2022} who consider analogues of modular square roots, Cilleruelo and Shparlinski~\cite{CS2013} who consider counting points on curves in arbitrary finite fields and Shparlinski and Zumalac{\'a}rregui~\cite{SZ2018} who consider Kloosterman fractions. With this in mind, we now let $q$ denote a prime power and $\F_q$ the finite field of order $q$. We let $\F_q[T]$ denote the ring of univariate polynomials over $\F_q$, with $\F_q(T)$ its field of fractions and $\Ki$ the analytic completion of $\F_q(T)$ at infinity. In this setting, the analogue of the condition~\eqref{eq:xjbounds} is for the variables to come from the set of polynomials of degree bounded by some parameter $n$, so that equations of the form~\eqref{eq:eqneqneqneqneqneqneqqneqnneqnefnqenqeneq} can be considered as counting solutions to equations in finite fields when the variables run through small dimensional linear subspaces.  We refer the reader to work of Sawin~\cite{10.1215/00127094-2020-0060} and Sawin and Shusterman~\cite{10.4007/annals.2022.196.2.1,Sawin2022}  for significant recent progress and applications of other short interval type problems in $\F_q[T]$.

For the problems we consider, the current state of art in $\F_q[T]$ is behind that of $\Z$. This is due to a lack of  general techniques (in $\Ki^d$) coming from the geometry of numbers which allow a lifting from congruences to equations. The purpose of this paper is to remedy this by providing results that connect the theory of lattices to counting solutions to congruences with variables from short intervals in $\F_q[T]$. As an application, we improve upon results of~\cite{BS2022,CS2013,SZ2018}, along with presenting results and techniques due to a number of different authors (over $\Z$) in a unified way. 

The study of the Geometry of numbers in $\Ki^d$ was initiated by Mahler~\cite{M1941} and has since been applied to various problems in diophantine approximation. Further developing this area is still a topic of interest. In this direction, we mention a recent paper of 
Roy and Waldschmidt~\cite{https://doi.org/10.1112/S0025579317000237} who generalise Schmidt and Summerer's parametric geometry of numbers~\cite{Schmidt2013-vw} into this setting. In Section~\ref{sec:Counting_and_localtoglobal}, we extend the estimate of Henk~\eqref{eq:Henk} into $\Ki^d$, which may be of independent interest.

\subsection{Overview}

 The structure of our paper is as follows:

 \begin{enumerate}
\item In Section \ref{sec:notation} we recall basic facts and notation related to $\F_q[T]$.
\item In Section~\ref{sec:results} present statements of new estimates for various counting problems which improve those of~\cite{BS2022,CS2013,SZ2018}.
\item Sections \ref{sec:misc_prelim} and \ref{sec:counting_global} present preliminaries that do not use any notions regarding lattices in $\Ki^d$. Section \ref{sec:misc_prelim} contains a few miscellaneous results inspired by \cite{BGKS}, and in Section \ref{sec:counting_global} we cite a number of preliminaries used for counting solutions to equations over $\F_q[T]$. As previously mentioned, the proofs of our results from Section~\ref{sec:results} all involve lifting congruences in $\F_q[T]/\langle F \rangle$ for some $F \in \F_q[T]$ to equations over $\F_q[T]$, after which these results from Section \ref{sec:counting_global} are applied. 
\item In Section \ref{sec:lattices} we summarise Mahler's work on the geometry of numbers in $\Ki^d$.
\item  Section \ref{sec:Counting_and_localtoglobal}, which can be regarded as the core of this paper,  develops a series of new results and techniques  useful for bounding intersections of lattices and convex bodies, as well as lifting from congruences to equations. These include a function field version of Henk's estimate~\eqref{eq:Henk}. As stated previously, in general these results are more precise than what is known over $\mathbb{R}^d$. This is mostly due to the fact that in $\Ki^d$, \textit{convex bodies are additive subgroups}. This means that the intersection of a lattice and convex body is really the intersection of two additive subgroups of $\Ki^d$, giving one more structure to work with. All of the techniques from this section are later used in proving the main results from Section \ref{sec:results}, although we present these techniques in greater generality than needed as we believe there is potential for many applications. For example, the ideas of Katznelson in \cite{Katznelson1993, Katznelson1994} regarding certain properties of integral matrices could be extended to $\F_q[T]$. 
\item  The rest of the paper is then devoted to proving the main results as stated in Section \ref{sec:results} by combining the results from Section~\ref{sec:Counting_and_localtoglobal} with Section~\ref{sec:counting_global}.
\end{enumerate}

\subsection{Notation}\label{sec:notation} Before outlining our main results, we first introduce some general notation. More notation of course will be introduced throughout the paper, and in particular our notations for lattices and convex bodies are introduced in Section \ref{sec:lattices}. Additionally, a notation guide is included in Appendix \ref{notationguide} for ease of reference. 

As before, we fix a prime power $q$ and let $\F_q$ denote the finite field of order $q$. We let $\F_q[T]$ denote the ring of univariate polynomials over $\F_q$ and $\F_q(T)$ its field of fractions. We define the usual absolute value $|\cdot |$ on $\F_q(T)$ as
\begin{align*}
   \left| \frac{g}{h}\right| =
    \begin{cases}
    q^{\deg g - \deg h}, & g \neq 0\\
    0, & g = 0.
    \end{cases}
\end{align*}
To be consistent with this definition, we will formally set $\deg 0 = -\infty$.  The completion of $\F_q(T)$ with respect to this absolute value is the field of Laurent series in $1/T$, 
$$\F_q((1/T)) = \left\{\sum_{i=-\infty}^na_iT^i~:~n \in \Z,~a_i \in \F_q,~a_n \neq 0\right\}.$$
Additionally, $|\cdot|$ extends to this space in the expected way:
$$\bigg{|} \sum_{i=-\infty}^na_iT^i \bigg{|} = q^n. $$
For simplicity of notation, from this point forward we set 
$$\Ki := \F_q((1/T)). $$
Given a positive integer $d$ we view $\Ki^d$ as a $d$-dimensional vector space over $\Ki$ and define a norm on this space by 
$$\|(x_1,...,x_d)\| = \max\left\{|x_1|,...,|x_d|\right\}. $$
Also given some $s = (s_1,...,s_d) \in \Ki^d$ and integer $m$, we define the open ball of radius $q^m$ in $\Ki^d$ centered at $s$ to be
\begin{align}\label{eq:openball}
    B_m(s) = \{(x_1,...,x_d) \in \Ki^d : |x_i-s_i| < q^m\}. 
\end{align}
For shorthand, we will set $B_m = B_m(0) \subseteq \Ki$. We will also let 
\begin{align*}
    \cI_m(s) = B_m(s) \cap \F_q[T]^d
\end{align*}
and $\cI_m = B_m \cap \F_q[T]$. 

 The variable $F$ will always denote a polynomial in $\F_q[T]$ of degree $r$. We will mostly be interested in certain congruences in $\F_q[T]/\langle F\rangle $, so we naturally identify $\cI_r$ as a set of representatives of $\F_q[T]/\langle F \rangle$. For any $x \in \F_q[T]$, by $\deg_F x $ we will mean the degree of the unique $x' \in \F_q[T]$ such that $\deg x' < r$ and $x' \equiv x \Mod{F}$. Additionally, given some $x \in \F_q[T]$ such that $\gcd(x,F) = 1$, we will denote by $\ov{x}$ the multiplicative inverse of $x$ modulo $F$. If this inverse is taken to a different modulus, this will be specified. 

\iffalse
If $F$ is irreducible, as fields, $\F_q[T]/F(T) \cong \F_{q^r}$. With this in mind, for an integer $m$ we write $x \sim m$ to mean  $x \in \F_{q^r}$, but $\deg x < m$ when we view $x$ as an element of $\F_q[T]/F(T)$. In particular, if $\rho$ is a root of $F$, then one can naturally identify the two sets
\begin{align*} 
\{x(T) & \in \F_q[T]/F(T):~x \sim m \} \\
& = \{u_0+u_1\rho + \ldots + u_{m-1} \rho^{m-1}:~  u_0, u_1, \ldots, u_{m-1} \in \F_q\}. 
\end{align*}  
Thus we switch freely between the languages of finite fields $\F_{q^r}$ and of function fields $\F_q[T]/F(T)$. 
Given $x\in \F_q[T]$ with $(x,F)=1$ we define $\ov{x}$ to be the unique polynomial of degree less than $\deg F$ such that $x\ov{x} \equiv 1 \Mod{F}$.
\fi

Given two functions $f,g$ on some additive group $G$, we define the convolution 
$$(f*g)(x)=\sum_{y \in G}f(y)g(x-y).$$
Let $f^{(1)}=f$ and for $k\ge 2$ inductively define 
$$f^{(k)}(x)=(f^{(k-1)}*f)(x).$$
Given a set $\cS,$ we also use $\cS(x)$ for the indicator function of $\cS$
$$\cS(x)=\begin{cases}1 \quad \text{if} \quad x\in \cS \\ 0 \quad \text{otherwise.} \end{cases}.$$
Under this notation, we note that 
\begin{align}\label{eq:kth_convolution}
    \cS^{(k)}(x) = |\{(s_1,...,s_k) \in \cS^k ~:~ s_1 + ... + s_k = x\}|.
\end{align}
 Given a complex weight $\balpha = \{\alpha(x)\}_{x \in G}$ and a real number $d \geq 1$ we define the $\ell_d$ norms in the usual way 
$$\|\alpha\|_d=\left(~\sum_{x \in G}|\alpha(x)|^{d}\right)^{1/d}.$$
We also define the support of $\balpha$
$$\supp(\balpha) = \{x ~:~ \alpha(x) \neq 0\}. $$

%% at the moment I dont think we use the stuff below, so ive commented it out for now, but we can put it back in if/when we use it
\iffalse
$\ell_2$ norms of convolutions are often referred to as additive energies. In general, for any set $\cS$, 

\begin{align*}
\|\cS^{(k)}\|_2^2& =|\{ s_1,\dots,s_{2k}\in \cS \ : \ s_1+\dots+s_k=s_{k+1}+\dots+s_{2k}\}|\\
&=:E_{k}(\cS).
\end{align*}
It will be convenient to generalise the above notation to other weight functions. In particular, we write
\begin{align*}
\|(\alpha*\alpha)\|_2^2=:E_{+}(\alpha).
\end{align*}
\fi

\section{Statements of results}

\label{sec:results}

\subsection{Points on polynomial curves in boxes } 
Given some finite set $\cS \subseteq \F_q[T]^2$ and some curve $\Phi(x,y) \in \F_q[T][x,y]$, we will first consider the problem of bounding 
$$C_{F, \Phi}\left(\cS\right) = \left|\{(x,y) \in \cS : \Phi(x,y) \equiv 0 \Mod{F}\}\right|$$
in a few special cases. 
This problem, and related problems, have been considered by a number of authors such as in \cite{CS2013, O2019, S2017} and of course, there are many results regarding similar questions over the integers. Here we utilize techniques outlined in \cite{MK2021}. 

We first have the following result. 
\begin{thm}\label{thm:curveinbox_1}
Suppose $\textup{char}(\F_q) > d$. Let $F \in \F_q[T]$ of degree $r$ and let $m \leq r$ denote a positive integer.  Let 
$$\Phi(x,y) = a_dx^d +...+ a_1x + a_0 - y\in \F_q[T][x, y]$$
with $a_d$ coprime to $F$, and let $s_1, s_2 \in \F_q[T]$. Then we have 
\begin{align*}
    C_{F, \Phi}(\cI_m(s_1, s_2)) 
    &\leq q^{m + 2m/(d^2+d) - 2r/(d^2+d) + o(m)} +  q^{m/d + o(m)}. 
\end{align*}
\end{thm}
In particular, we note that if $m < 2r/(d^2+1)$ then this implies 
\begin{align}\label{eq:curveinbox_1}
    C_{F, \Phi}(\cI_m(s_1, s_2)) 
    \leq  q^{m/d + o(m)}. 
\end{align}
 In \cite[Theorem 2]{CS2013} the bound 
\begin{align}\label{eq:shpar_cill}
    C_{F, \Phi}(\cI_m(s_1, s_2))  &\leq q^{m-m/2^{d-1}+o(m)} + q^{m-(r-m)/2^{d-1}+ o(m)}
\end{align}
is given, but they remark that perhaps one may expect the bound (\ref{eq:curveinbox_1}). Theorem {\ref{thm:curveinbox_1}} always improves upon (\ref{eq:shpar_cill}) for $d\geq 5 $ and when $m< 2r/5, 2r/5$ or $4r/9$ when $d=2,3$ or $4$, respectively. 

A special case of our next result is equivalent to counting points on elliptic curves in certain shifted subspaces of residue rings. We again utilize techniques outlined in \cite{MK2021}. 
\begin{thm}\label{thm:curveinbox_2}
Suppose $\textup{char}(\F_q) > 3$. Let $F \in \F_q[T]$ of degree $r$ and let $m\leq r$ denote a positive integer. Let 
$$\Phi(x,y) = a_3x^3 + a_2x^2 + a_1x + a_0 - y^2 \in \F_q[T][x, y]$$  with $a_3$ coprime to $F$, and  let $s_1, s_2 \in \F_q[T]$. Then for any $\epsilon > 0$, if $m < r(1/2 - \epsilon)$ we have
\begin{align*}
  C_{F, \Phi}(\cI_m(s_1, s_2))
   &\leq q^{3m/2-r/6 + o(m)} +  q^{m/3 + o(m)},   
\end{align*}
where the $o(m)$ term may depend on $\epsilon$. 
\end{thm}

Lemma~\ref{lem:coppersmith} implies a trivial bound of $O(q^m)$. Theorem \ref{thm:curveinbox_2} improves  upon this for $m < r/3$. 

\begin{rem}
Both Theorems \ref{thm:curveinbox_1} and \ref{thm:curveinbox_2} rely on an analogue of Vinogradov's mean value theorem given by T. Wooley in \cite{W2019}. This result, as is, requires large characteristic which is why there is a restriction on characteristic in our results currently. Although, it seems work is currently in progress by T. Wooley and Y. R. Liu to allow this analogue of Vinogradov's mean value theorem to work in any characteristic. This would immediately allow our results to hold for any characteristic. 
\end{rem}

\subsection{Kloosterman equations}
Next, given some positive integer $k$ and some finite $\cS \subseteq \F_q[T]$ we let 
\begin{align}
\label{eq:inv-def}
&E_{F, k}^\inv(\cS)\\ 
&~~= \left|\{(x_1,...,x_k) \in \cS^k : \ov{x_1} + ... + \ov{x_k} \equiv \ov{x_{k+1}} + ... + \ov{x_{2k}} \Mod{F}\}\right|. \nonumber
\end{align}

Our first result for modular inverses may be considered a function field version of a bound due to Bourgain and Garaev~\cite{BourgainGaraev2014}.
\begin{thm}\label{thm:sums_of_inverses}
Let $F \in \F_q[T]$ of degree $r$. Let $m$ and $k$ be positive integers with $k$ fixed and $m \leq r$. Then 
$$E_{F,k}^\inv\left(\cI_m\right) \leq q^{km + o(m)} + q^{m(3k-1)-r + o(m)}. $$
\end{thm}

We note that this improves upon the trivial bound  $q^{m(2k-1)}$ when $m < r/k$. But, in the case of $k=2$ and $F$ irreducible, this never improves upon 
\begin{align}\label{eq:old_energy_inverses_bound}
    E_{F,2}^\inv\left(\cI_m\right) \leq q^{7m/2 - r/2 + o(m)} + q^{2m + o(m)} 
\end{align}
given in \cite[Theorem 2.2]{BS2022}. 

Next, the question of bounding $E^\inv_{F,k}\left( \cI_m(s) \right)$ is considered in~\cite{SZ2018}, where they achieve a bound stronger than a direct extension of what is best known over the integers \cite[Theorem 4]{BourgainGaraev2014}. We can go further by improving upon~\cite[Theorem~1.1]{SZ2018}. 

\begin{thm}\label{thm:sums_of_inverses-general}
Let $F \in \F_q[T]$ of degree $r$ be irreducible and $s\in \F_q[T]$ be arbitrary. Let $m$ and $k$ be positive integers with $k \geq 2$ fixed and 
\begin{align}
\label{eq:kloostermanpar}
m<  \left(1-\frac{3k}{2r} \right)\frac{r}{4k^2-4k+5}.
%\left(\frac{2k-2}{2k-1}\right)\frac{r}{4k^2-3k-1} 
\end{align}
Then 
$$E^\inv_{F,k}\left( \cI_m(s) \right) \leq q^{km+o(m)}. $$
\end{thm}

In the case of $k=2$, ~\eqref{eq:old_energy_inverses_bound} can be generalized to arbitrary intervals, and this would always be better than Theorem \ref{thm:sums_of_inverses-general}. But for $k > 2$, Theorem \ref{thm:sums_of_inverses-general} directly improves upon \cite[Theorem 1.1]{SZ2018} and of course is always non-trivial.

\subsection{Sums of modular square roots }
%In \cite[Theorem 2.1]{BS2022}, it is shown that 
%\begin{align*}
  %  E_{q,r}^{\sqrt}(\balpha) \leq \|\balpha\|_1^2\|\balpha\|_\infty^2q^{m/2 + o(m)}\big{(}q^{m-r/2}+1\big{)}
%\end{align*}
%which in particular implies 
%\begin{align}\label{eq:old_energy_bound}
%    E_{q,r}^{\sqrt}(\bm{1}_m) \leq q^{o(m)}\big{(}q^{7m/2-r/2}+q^{5m/2}\big{)}
%\end{align}
Throughout this section,  we will let 
$$\balpha = (\alpha(x))_{x \in \F_q[T]/\langle F \rangle}$$
denote a sequence of complex weights. Although formally this will be a sequence of weights on $\F_q[T]/\langle F \rangle$ for the sake of norms, we can of course define the value of $\alpha(x)$ for any $x \in \F_q[T]$ by composing with the canonical quotient map $\F_q[T] \to \F_q[T]/\langle F \rangle$. With this in mind, given a positive integer $m \leq r$, we will have $\balpha$ satisfy 
\begin{align}\label{Weights}
    \textup{supp}(\alpha) \subseteq \{\deg_Fx < m\}.
\end{align}
As a special case we will set $\bm{1}_{m}$ as the characteristic function of the set $\{\deg_Fx < m\}$. Now for any $\balpha$ as above, we define 
 $$ E_{F,k}^{\sqrt}(\balpha) = \smashoperator[r]{\sum_{\substack{(x_1,...,x_{2k}) \in \cI_r^k \\  x_1+...+x_k \equiv x_{k+1}+...+x_{2k} (F)}}}\:\:\:\:\alpha({x_1^2})\ov{\alpha({x_2^2})}\cdots\alpha({x_{2k-1}^2})\ov{\alpha({x_{2k}^2})}.$$
Perhaps the most interesting scenario is 
 \begin{align*} 
E_{F,k}^{\sqrt}(\bm{1}_m)=  \big{|}\{(x_1,...,x_{2k}) \in \cI_r^{2k}: ~  x_1 + ... + x_k &\equiv x_{k+1} +...+x_{2k} , \\
 & \qquad   \deg_F(x_i^2) < m\}\big{|}. 
\end{align*}  

We have borrowed this notation from \cite{BS2022, DKSZ2020, KSSZ2021}, where this quantity is referred to as the \textit{additive energy of modular roots} in the case of $k=2$. Non-trivial bounds for $E_{F, 2}^{\sqrt}(\balpha)$ are given in \cite{BS2022}, and analogues of this problem over the integers have been investigated in \cite{DKSZ2020, KSSZ2021}. Taking inspiration from these, we have the following.

\begin{thm}\label{thm:energy_bound}
Let $q$ be odd and $F \in \F_q[T]$ be irreducible of degree $r$. For any integer $m \leq r$ and a weight $\balpha$ on $\F_q[T]/\langle F \rangle$ as in (\ref{Weights}) we have 
$$
 E_{F, 2}^{\sqrt}(\balpha) \leq \|\balpha\|_1^{4/3}\|\balpha\|_\infty^{8/3}q^{m + o(m)}(q^{7m/6 - r/2} + 1).
 $$
\end{thm} 
We note that this implies 
$$E_{F,2}^{\sqrt}(\bm{1}_m) \leq q^{o(m)}\big{(}q^{7m/2-r/2} + q^{7m/3} \big{)}. $$
This is an improvement upon \cite[Theorem 2.1]{BS2022}, but a more specialized argument can further yield the following. 
\begin{thm}\label{thm:energy_1m_improved}
Let $q$ be odd and $F \in \F_q[T]$ of degree $r$. For any integer $m \leq r$ we have
$$E_{F,2}^{\sqrt}(\bm{1}_m) \leq q^{o(m)}\big{(}q^{7m/2-r/2} + q^{2m} \big{)}. $$
\end{thm}
Analogous to \cite{KSSZ2021} we also consider the case $k=4$ when $\balpha = \bm{1}_m$. It is simple to show that 
$$ E_{F,4}^{\sqrt}(\bm{1}_m) \leq q^{4m}E_{F,2}^{\sqrt}(\bm{1}_m)$$
which, by Theorem \ref{thm:energy_1m_improved}, implies 
\begin{align}\label{eq:T4_trivial}
    E_{F,4}^{\sqrt}(\bm{1}_m) \leq  q^{4m + o(m)}\big{(}q^{7m/2-r/2} + q^{2m} \big{)}.
\end{align}
We improve upon this for small $m$ with the following.
\begin{thm}\label{thm:T4_bound}
Let $q$ be odd and $F \in \F_q[T]$ be irreducible of degree $r$. For any integer $m \leq r$ we have
$$E_{F,4}^{\sqrt}(\bm{1}_m) \leq q^{6m + o(m)}\big{(}{q^{11m/2-r/2}}+{q^{3m-r/4}} + q^{5m/8-r/8}\big{)} + q^{5m + o(m)}. $$
\end{thm}
This is an improvement upon (\ref{eq:T4_trivial}) if $m < r/12$.

\section{Miscellaneous Preliminaries }\label{sec:misc_prelim}
In this section, we provide several miscellaneous preliminaries that are essential for proving some of our main results. These results mainly concern certain properties of polynomials in $\F_q[T][x]$.
\subsection{Heights of polynomials  }
Given a polynomial $P(x)=a_0+a_1x+\dots+a_dx^{d}\in \F_q[T][x]$ we define 
$$h(P)=\max_{0\le j \le d}\deg{a_j},$$
and refer to $h(P)$ as the \textit{height} of $P$.
\begin{lemma} 
\label{lem:height-additive}
For any $P_1,P_2\in \F_q[T][x],$
$$h(P_1P_2)=h(P_1)+h(P_2).$$
\end{lemma}
\begin{proof}
Suppose 
$$P_1(x)=a_0+\dots+a_dx^{d},$$
and 
$$P_2(x)=b_0+\dots+b_e x^{e}$$
for some positive integers $d$ and $e$ and $a_i, b_i \in \F_q[T]$. If we define the sets 
$$\cI=\{ 0\le i \le d \ : \ \deg a_i =h(P_1)  \},$$
$$\cJ=\{ 0\le j \le e \ : \ \deg b_j =h(P_2)  \},$$
then this allows one to partition the coefficients of $P_1$ and $P_2$ as 
$$P_1=P^{(1)}_1+P^{(2)}_1,$$
and 
$$P_2=P^{(1)}_2+P^{(2)}_2,$$
with 
\begin{align*}
P^{(1)}_1(x)=\sum_{i\in \cI}a_ix^{i}, \quad P^{(1)}_2(x)=\sum_{j\in \cJ}b_jx^{j}
\end{align*}
and of course with $P_i^{(2)} = P_i - P_i^{(1)}$.
Note that 
\begin{align*}
h(P^{(1)}_1P^{(1)}_2)=h(P_1)+h(P_2),
\end{align*}
and 
\begin{align*}
 h(P^{(i)}_1P^{(j)}_2)<h(P_1)+h(P_2) \ \ \text{if} \ \  (i,j)\neq (1,1),
\end{align*}
from which the result follows after noting 
$$P_1P_2=(P^{(1)}_1+P^{(2)}_1)(P^{(1)}_2+P^{(2)}_2).$$
\end{proof}
\begin{lemma}
\label{lem:factorisation-bound}
Let  $P(x)=a_0+a_1x+\dots+a_dx^{d}\in \F_q[T][x]$ with each $a_j$ satisfying 
$$|a_j|\le q^{(d-j+a)\ell},$$
for some $\ell\in \mathbb{N}$ and $a\in \mathbb{Z}$.
If there exists $P_1,P_2\in  \F_q[T][x],$ satisfying
$$P=P_1P_2,$$
then 
$$P_1(x)=b_0+b_1x+\dots+b_ex^{e},$$
for some $e\le d$ and $b_j \in \F_q[T]$ satisfying 
$$|b_j|\le q^{(e-j+a)\ell}.$$
\end{lemma}
\begin{proof}
Consider the polynomial 
$$P^*(x)=P(T^{\ell}x)=P_1(T^{\ell}x)P_2(T^{\ell}x)=P_1^*(x)P^*_2(x).$$
We see that 
$$h(P^*)\le \ell(d+a).$$
Since 
$$h(P_2^*)\ge \deg{P_2}\ell,$$
Lemma~\ref{lem:height-additive} now implies 
\begin{align*}
h(P_1^*)\le (d-\deg{P_2}+a)\ell=(\deg{P_1}+a)\ell,
\end{align*}
and hence 
$$\deg(b_j)+\ell j \le (\deg{P_1}+a)\ell,$$
from which the result follows.
\end{proof}
\subsection{Resultants }
Given an integral domain $R$ and two polynomials $P,Q\in R[x]$ of the form
$$P(x)=a_0+\dots+a_dx^{d},$$
and 
$$Q(x)=b_0+\dots+b_e x^{e},$$
we define the Sylvester matrix 
$$S(P,Q) = \left(\begin{array}{c}  A\\  B\\ \end{array}\right) $$
where
\begin{equation*}
A=\left(
  \begin{array}{cccccccc}    a_{d} &  \ldots &  a_{1} &a_{0} & 0 & 0 & \ldots & 0 \\
    0 & a_{d}&  \ldots &  a_{1} &a_{0} & 0 & \ldots & 0  \\
    \vdots & \vdots & \ddots  & \vdots & \vdots & \vdots & \ddots & \vdots \\
    0 & \ldots & 0 & a_{d} &  \ldots & \ldots &  a_{1} &a_{0}\\
  \end{array}
\right)
\end{equation*}
and
\begin{equation*}
B=\left(
  \begin{array}{cccccccc}   b_{e} & \ldots & b_{1}  &  b_{0}  & 0 & 0 & \ldots & 0 \\
    0 & b_{e} & \ldots & b_{1}  &  b_{0} & 0 & \ldots & 0  \\
    \vdots & \vdots & \ddots  & \vdots & \vdots & \vdots & \ddots & \vdots \\
    0 & \ldots & 0 & b_{e} & \ldots & \ldots & b_{1}  &  b_{0}\\
  \end{array}
\right).
\end{equation*}
The resultant of $P$ and $Q$ is then defined as 
$$
\Res(P, Q)=\det S(P,Q).
$$
The main (equivalent) properties of resultants we will use are the following, which are well known. 
\begin{lemma}
\label{lem:resultant}
For any polynomials $P,Q\in R[x]$, 
\begin{itemize}
    \item $\Res(P,Q)=0$
if and only if $\gcd(P,Q) \neq 1$ in $R[x]$. 
\item $\Res(P,Q)=0$
if and only if $P$ and $Q$ have a common root over $\overline \F$, where $\overline \F$ is an algebraically closed field containing $R$.  
\end{itemize}
%$P$ and $Q$ have a common root over $\overline \F$, where $\overline \F$ is an algebraically closed field containing $R$. 
\end{lemma}

We next present a construct of~\cite[Section~2]{BGKS} which is useful for bounding resultants. Let $m,n\ge 2$ be positive integers and $\sigma\in\R$. Define the $(n-1) \times (m+n-2)$
circulant matrix   $A(m, n,\sigma)$  by
\begin{equation*}
\left(
  \begin{array}{cccccccc}
    \sigma & \sigma+1 & \ldots & \sigma + m-1 & 0 & 0 & \ldots & 0 \\
    0 & \sigma  & \ldots  & \sigma +m-2 &\sigma + m-1 & 0 & \ldots & 0  \\
    \vdots & \ldots & \ldots  & \ldots & \ldots & \ldots & \ldots & \vdots \\
    0 & \ldots & 0 & \sigma &\sigma+1 & \ldots  & \ldots &\sigma + m-1\\
  \end{array}
\right).
\end{equation*}
We call the $(i,j)$-th element of $A(m, n,\sigma)$ \textit{marked} if $i \leq j \leq i+m-1$. If $\sigma > 0$ then being \textit{marked} is equivalent to being non-zero. The following is~\cite[Lemma~1]{BGKS}.
\begin{lemma}
\label{lem:DeterMagic} Let $m,n\ge 2$ be integers and
$\sigma,\vartheta\in\R$. Let $x_{i,j}$ denote the $(i,j)$-th entry of the  $(m+n-2)\times (m+n-2)$ matrix
$$
X(m, n)=\left(
  \begin{array}{c}
    A(m,n,\sigma)\\
    A(n,m,\vartheta) \\
  \end{array}
\right).
$$
For any permutation 
$$\pi:\{1,\dots,m+n-2\}\rightarrow \{1,\dots,m+n-2\},$$
if $x_{j, \pi(j)}$ is marked for all $j$
then
$$
\sum_{j=1}^{m+n-2}x_{j,\pi(j)}
=(m-1+\sigma)(n-1+\vartheta) - \sigma\vartheta.
$$
\end{lemma}

Lemma~\ref{lem:DeterMagic} allows for a straightforward adaptation of~\cite[Corollary~3]{BGKS} into our setting. 
\begin{lemma}
\label{lem:resultantbound}
Let $d$ and $e$ be positive integers, and let $\ell, a, b \in \mathbb{R}$. Let $P,Q\in \F_q[T][x]$ be polynomials
\begin{align*}
P(x)=a_0+\dots+a_{e-1}x^{e-1}, \quad Q(x)=b_0+\dots+b_{d-1}x^{d-1}
\end{align*}
satisfying 
\begin{align*}
|a_j|\le q^{(e-j+a)\ell}, \quad |b_j| \le q^{(d-j+b)\ell}.
\end{align*}
Then 
\begin{align*}
|\Res(P,Q)| \le q^{((e+a)(d+b)-(a+1)(b+1))\ell}.
\end{align*}
\end{lemma}

\begin{proof}
The proof is essentially identical to the proof of \cite[Corollary~3]{BGKS}, but we sketch the details here. Let  $x_{i,j}$ denote the $ij$-th element of $S(P,Q)$. Then if by $\pi$ we denote a permutation 
$$\pi: \{1,...,m+n-2\} \to \{1,...,m+n-2\} $$
we have 
\begin{align*}
    \log_q|\Res(P,Q)| 
    &\leq \log_q\max_{\pi}\left|\prod_{j=1}^{d+e-2}x_{j, \pi(j)} \right|\\
    &= \max_{\pi}\sum_{j=0}^{d+e-2} \deg x_{j, \pi(j)}.
\end{align*}
The result now follows from Lemma \ref{lem:DeterMagic} with $\sigma = a+1, ~\vartheta = b+1, m=e$ and $n = d$. 
\end{proof}

\begin{cor}
\label{cor:good}
Let $F\in \F_q[T]$ be an irreducible polynomial of degree $r$ and let $d,e$ and $\ell$ be positive integers with $e \leq d$. Suppose $P,Q_1,\dots,Q_m\in \F_{q}[T][x]$ are linearly independent over $\F_q(T),$ with $P$ irreducible and suppose there exists $s\in \F_{q}[T]$ and $a,b \in \R$ such that
\begin{align}
\label{eq:s-root}
 \quad P(s)\equiv Q_1(s)\equiv \dots \equiv Q_m(s)\equiv 0 \Mod{F}
\end{align}
\begin{align}
\label{eq:Pcoefficientbound}
P(x)=a_0+\dots+a_{e-1}x^{e-1}, ~~|a_j|\le q^{(e-j+a)\ell},
\end{align}
and 
\begin{align}
\label{eq:Qcoefficientbound}
Q_i(x)=b_{i,0}+\dots+b_{i,d_i-1}x^{d_i-1}, ~~ |b_{i,j}|\le q^{(d_i-j+b)\ell}, ~~ d_i\le d.
\end{align}
Further suppose
\begin{align}
\label{eq:deabconds}
((e+a)(d+b)-(a+1)(b+1))\ell<r.
\end{align}
Then 
\begin{align*}
e\le d-m.
\end{align*}
\end{cor}
\begin{proof}
By Lemma~\ref{lem:resultant} and equation \eqref{eq:s-root} we obtain
\begin{align}
\label{eq:Res1}
\Res(P,Q_i)\equiv 0 \Mod{F}, \quad 1\le i \le m.
\end{align}
By Lemma~\ref{lem:resultantbound}, and  equations~\eqref{eq:Pcoefficientbound} and~\eqref{eq:Qcoefficientbound} we can say
\begin{align*}
\Res(P,Q_i)\le q^{((e+a)(d+b)-(a+1)(b+1))\ell}. 
\end{align*}
Thus by~\eqref{eq:deabconds} and~\eqref{eq:Res1}, we can conclude
\begin{align*}
\Res(P,Q_i)=0.
\end{align*}
Since $P$ is irreducible, another application of Lemma~\ref{lem:resultant} implies 
\begin{align*}
P|Q_i, \quad 1\le i \le m.
\end{align*}
Since $P,Q_1,\dots,Q_m$ are linearly independent and $$\deg{P},\deg{Q_i}\le d-1,$$ we see that 
\begin{align*}
\dim_{\F_q(T)}\{ \  Q\in \F_q(T)[x] \ : \deg{Q}\le d-1, \ \ P|Q \  \}\ge m+1,
\end{align*}
and the result follows since 
$$\dim_{\F_q(T)}\{ \  Q\in \F_q(T)[x] \ : \deg{Q}\le d-1, \ \ P|Q \  \}= d-\deg{P}=d-e+1.$$
\end{proof}

\subsection{A recursive inequality }
The final result of this section is essentially contained in the proof of~\cite[Lemma~15]{BGKS}. 
\begin{lemma}
\label{lem:recursive}
Let $F\in \F_q[T]$ have degree $r$. Let $s \in \F_q[T]$ and suppose $\gcd(s,F)=1$. Further suppose $\ell$ is a positive integer and $u_0,\dots,u_{d-1} \in \F_q[T]$ are not all zero modulo $F$ and satisfy
\begin{align}
\label{eq:uju0}
u_0s^{j}\equiv u_j \Mod{F}, \quad |u_j|\le q^{\ell((d^2-(2\beta+1)d+2)/2 + j) + d},
\end{align}
for some $d \geq 2$ satisfying
\begin{align}
\label{eq:alphaconds}
\ell  < \frac{r}{d^2-(2\beta+1)d+5}\left(1-\frac{3d}{2r} \right). 
\end{align}
Then there exists $a$ and $b$, both non-zero modulo $F$,  satisfying $bs\equiv a \mod{F}$ and 
\begin{align*}
\quad |a|\leq q^{\ell d(d+1-2\beta)/(2d-2) + d/(d-1)}, \quad |b|\leq q^{\ell(d^2-(2\beta+1)d+2)/(2d-2) + d/(d-1)}.
\end{align*}
\end{lemma}
\begin{proof}
We will first prove by induction that there exists coprime $a,b\in \F_q[T]$ such that for each $2\le j \le d-1$ there exists $r_j\in \F_q[T]$ such that 
\begin{align}
\label{eq:u01j}
u_0=r_jb^{j}, \quad u_1=r_jab^{j-1}, \quad u_j=r_j a^{j}.
\end{align}

Our base case is  $j=2$. We note that~\eqref{eq:uju0} and~\eqref{eq:alphaconds} imply that 
\begin{align*}
u_0u_2=u_1^{2}.
\end{align*}
Hence there exists $r_2,a,b\in \F_q[T]$ with $\gcd(a,b) = 1$ satisfying 
\begin{align*}
u_0=r_2b^2, \quad u_2=r_2a^2, \quad u_1=r_2ab.
\end{align*}
%%we have $u_1s \equiv u_2s^2$ so $u_1 \equiv u_2s$. Thus $u_0u_2 \equiv u_1su_2 \equiv u_1^2. $
Now for the inductive step, assume that for some integer $k < d-1$ we have that for all $j \leq k$, ~\eqref{eq:u01j} holds.  From~\eqref{eq:uju0}, we have 
\begin{align*}
u_1^{k+1}\equiv u_{k+1}u_0^{k} \Mod{F}.
\end{align*}
By~\eqref{eq:u01j} (with $j=k$), this implies 
\begin{align}
\label{eq:Fmod}
r_{k}a^{k+1}\equiv u_{k+1}b \Mod{F}.
\end{align}
Another application of our inductive hypothesis combined with~\eqref{eq:uju0} and~\eqref{eq:alphaconds} implies 
\begin{align*}
   |r_k a^{k+1}|\le |r_ka^{k}|^{(k+1)/k}\le |u_k|^{(k+1)/k}&\le q^{\ell((d^2-(2\beta+1)d+2)/2+d/\ell + k)(k+1)/k}\\
   &\le q^{\ell((d^2-(2\beta+1)d+2)/2+d/\ell + 2)3/2}\\
   &\leq q^{\ell(d^2-(2\beta+1)d+5) + \frac{3d}{2}}<q^{r}.
\end{align*}
A similar calculation shows that
\begin{align*}
|u_{k+1}b| <q^{r}.
\end{align*}
Combining the above with~\eqref{eq:Fmod}, we see that
$$
r_ka^{k+1}= u_{k+1}b,$$
and hence 
$$u_{k+1}=r_{k+1}a^{k+1}$$ 
with $r_{k+1}=r_k/b$ (and note that $r_{k+1}\in \F_{q}[T]$ since $a$ and $b$ are coprime). Substituting into~\eqref{eq:u01j} with $j=k$ this completes the inductive step. 

Applying~\eqref{eq:u01j} with $j=d-1$, by~\eqref{eq:uju0} we have 
\begin{align*}
bs\equiv a \Mod{F},  
\end{align*}
and 
\begin{align*}
&|a|\le |u_{d-1}|^{1/(d-1)}\le q^{\ell d(d+1-2\beta)/(2d-2) + d/(d-1)},\\
&|b|\le |u_0|^{1/(d-1)}\le q^{\ell(d^2-(2\beta+1)d+2)/(2d-2) + d/(d-1)},
\end{align*}
which completes the proof.

\end{proof}

\section{Counting solutions to equations over \texorpdfstring{$\F_q[T]$}{} }\label{sec:counting_global}

The main focus of this paper is to develop techniques useful for counting solutions to congruences in $\F_q[T]/\langle F\rangle $ and often this is achieved by lifting congruences modulo $F$ to equations over $\F_q[T]$.  Hence we require some preliminary results regarding counting solutions to equations over $\F_q[T]$, which are presented in this section. 

First we have the following very important bound from \cite[Lemma 1]{CS2013} which will be used in multiple proofs. 
\begin{lemma}
\label{lem:divisorbound}
For any $x\in \F_q[T]$, 
%HERE
 the number of divisors of $x$ is $q^{o(\deg x)}$.
\end{lemma}
 This is analogous to the classical bound on the divisor function $\tau(n) \ll n^{o(1)}$. 
 
Next, the following is~\cite[Lemma 2.6]{SZ2018}.
\begin{lem}\label{lem:SOI:inverses_equality}
For a fixed positive integer $k$ and $\beta\in \overline{\F}_{q}(T)$, the number of solutions to 
\begin{align*}
    \frac{1}{x_1+\beta}+...+\frac{1}{x_k+\beta} = \frac{1}{x_{k+1}+\beta}+...+\frac{1}{x_{2k}+\beta}
\end{align*}
with $(x_1,...,x_{2k}) \in \cI_m^{2k}$ is bounded above by 
$q^{km + o(m)}.$
\end{lem}

We also require the following result from \cite[Theorem 1]{S2017}.
\begin{lem}\label{lem:Bombieri_pila}
Fix a positive integer $d$. Let $\Phi(x,y) \in \F_q[T][x,y]$ have degree $d$ and be irreducible over $\F_q[T]$. Then the number of solutions to 
$$\Phi(x,y) = 0, ~~(x,y) \in \cI_m^2$$
is bounded above by $q^{m/d + o(m)}$. 
\end{lem}

The rest of this section is now dedicated to applying certain orthogonality relations that allow us to count solutions to certain systems of equations. 

It is well established that $\Ki$ is locally compact. One way to see this is that the valuation ring of $\Ki$ is $\F_q[[1/T]]$, whose unique maximal ideal is $\langle 1/T \rangle$. Now we have that $\F_q[[1/T]]/\langle 1/T \rangle \cong \F_q$, which implies that $\Ki$ is locally compact by \cite[Ch. 2, Prop. 1]{S2013}.
This now implies, by \cite{H1933}, that there exists a unique Haar measure $\mu$ on $\Ki$, which we normalize such that $\int_{B_0} 1 d\mu(\alpha) = 1$. Recall that a Haar measure is invariant under translation. 

For any $x \in \Ki$ if we have 
$$x = \sum_{i= -\infty}^{n}a_iT^i $$
for simplicity we will write $[x]_{-1} := a_{-1}$. We thus define the additive character on $\Ki$, 
$$e(x) = \exp\bigg{(}\frac{2\pi i}{p}\Tr[x]_{-1} \bigg{)} $$
where $\Tr$ is the absolute trace of $\F_q$ and $p$ is the characteristic of $\F_q$. This character satisfies the orthogonality relation (see \cite[Theorem 3.5]{H1966})
\begin{align*}
    \int_{B_0}e(x\alpha)d\mu(\alpha) = 
\begin{cases}
0, &\text{ if } |x| \geq 1\\
1, &\text{ if }|x| < 1
\end{cases}
\end{align*}
which in particular implies (as stated in \cite[Equation 17.1]{W2019})
\begin{align}\label{eq:fourier_orthogonality}
    \int_{B_0}e(x\alpha)d\mu(\alpha) = 
\begin{cases}
0, &\text{ if } x \in \F_q[T]\setminus\{0\}\\
1, &\text{ if }x = 0.
\end{cases}
\end{align}
The following is a special case of \cite[Theorem 17.1]{W2019}, and can be regarded as a strong version of Vinogradov's mean value theorem for function fields. 
\begin{lem}\label{lem:wooley_VMVT}
Let $s$ and $k$ be fixed positive integers such that $s \leq k(k+1)/2$ and suppose $\mathcal{S} \subseteq \cI_{n}$ for a positive integer $n$. Provided that $\textup{char}(\F_q) > k$ one has 
$$\int_{B_0^k}\bigg{|}\sum_{x \in \mathcal{S}}e(\alpha_1x + ... + \alpha_kx^k) \bigg{|}^{2s}d\mu(\alpha_1)...d\mu(\alpha_k) \leq q^{o(n)}|\mathcal{S}|^s. $$
This immediately implies, by (\ref{eq:fourier_orthogonality}), that 
\begin{align*}
    \big{|}\big{\{} (x_1,..&.,x_{2s}) \in \cS^{2s} : \\
    &x_1^j + ... +x_{s}^j = x_{s+1}^j + ... + x_{2s}^j, ~1 \leq j \leq k\big{\}}\big{|} \leq ~q^{o(n)}|\cS|^s.
\end{align*}
\end{lem}
For any $\lambda = (\lambda_1,...,\lambda_k) \in \F_q[T]^k$ we define $J_{\lambda,k,s}(\cS)$ to count the number of solutions to the system 
 \begin{align}\label{eq:J_system}
     x_1^j + ... + x_{2s}^j = \lambda_j, ~ 1\leq j \leq k, ~x_i \in \cS.
 \end{align}
The following is now implied by the previous result. 
\begin{cor}
\label{cor:wooley}
Let $\cS \subseteq \cI_n$ for some positive integer $n$. For fixed positive integers $s$, $k$  and $\lambda \in \F_q[T]^k$, if $\textup{char}(\F_q) > k$ and $s \leq k(k+1)/2$ one has 
$$J_{\lambda,k,s}(\cS) \leq q^{o(n)}|\mathcal{S}|^s. $$
\end{cor}
\begin{proof}
Let $\lambda = (\lambda_1, ..., \lambda_k)$. By~equation~(\ref{eq:fourier_orthogonality}),
\begin{align*}
   &J_{\lambda,k,s}(\cS) \\
   &~~= \sum_{x_1,...,x_{2s} \in \mathcal{S}}\int_{B_0^d}e\bigg{(}\sum_{j=1}^d\alpha_j\bigg{(}\sum_{i=1}^s(x_i^j+x_{s+i}^j) - \lambda_j\bigg{)}\bigg{)}d\mu(\alpha_1)...d\mu(\alpha_d)\\
   &~~= \int_{B_0^d}\bigg{(}\sum_{x \in \cS}e(\alpha_1x + ... + \alpha_dx^d)\bigg{)}^{2s}e\bigg{(}-\sum_{j=1}^d\alpha_j\lambda_j \bigg{)}d\mu(\alpha_1)...d\mu(\alpha_d)\\
   &~~\leq \int_{B_0^d}\bigg{|}\sum_{x \in \cS}e(\alpha_1x + ... + \alpha_dx^d)\bigg{|}^{2s}d\mu(\alpha_1)...d\mu(\alpha_d)
\end{align*}
and now the result follows by Lemma \ref{lem:wooley_VMVT}.
\end{proof}

\section{Background on Lattices }
\label{sec:lattices}

\subsection{Mahler's results } 
We will now provide the necessary preliminaries regarding lattices and convex bodies in $\Ki^d$, and the rest of the paper will build upon and apply the ideas presented here. As previously mentioned, it was Mahler \cite{M1941} who initially developed these ideas, as an analogue to Minkowski's results \cite{M1910}. 

We fix a positive integer $d$ and define an $\F_q[T]$-lattice $\cL$ in $\Ki^d$ to be an  $\F_q[T]$-submodule of $\Ki^d$ of the form 
$$\cL = \text{span}_{\F_q[T]}\{v_1, ..., v_d\} $$
where $\{v_1, ..., v_d\}$ is an $\Ki$-basis of $\Ki^d$. Equivalently, a lattice is of the form 
$$\cL= A_\cL\F_q[T]^d $$
for some invertible $A_\cL \in \text{GL}(\Ki^d)$. For ease of notation we write 
$$\det \cL = |\det A_\cL|. $$

We similarly define a convex body $\cB$ in $\Ki^d$ to be a set of the form 
$$\cB = U_\cB B_1^d $$
for some $U_\cB \in \text{GL}(\Ki^d)$ and where $B_1 = B_1(0)$ is as in (\ref{eq:openball}).  Note that $\cB$ forms a $B_1$-module, so in particular a convex body is closed under addition. This is different than the initial definition used by Mahler in \cite[page 491]{M1941}, but he proves that these are equivalent \cite[page 498]{M1941}. We then define 
$$\vol \cB = |\det U_\cB|. $$
Although we measure volume here using determinants, \cite{C2012} shows that this can all be equivalently framed in terms of the Haar measure $\mu$ as introduced in Section \ref{sec:counting_global}.
Next, given such a convex body $\cB$ we associate to it a norm function $N_\cB$ given by 
$$N_\cB(x) = \|U_\cB^{-1}x\|. $$
Note that 
\begin{align*}
    \cB = \{x \in \Ki^d: N_\cB(x) \leq 1\}.
\end{align*}
We could also write 
\begin{align}\label{eq:normfunction_inf}
N_\cB(x) 
    &= \min_{\substack{\xi \in \Ki \\ \|U_\cB^{-1}x\| \leq |\xi|}} |\xi|
    = \min_{\substack{\xi \in \Ki\\\ U_\cB^{-1}x \in \xi B_1}}|\xi|
    ~= \min_{\substack{ \xi \in \Ki  \\  x \in \xi\cB}}|\xi|
\end{align}
which is analogous to definitions that typically appear when dealing with lattices and convex bodies in Euclidean space.

We now define the successive minima of a lattice $\cL$ with respect to a convex body $\cB$. Let 
$$\sigma_1 = \min_{\substack{x \neq 0 \\ x \in \cL}}N_\cB(x) $$
and let $x^{(1)}$ be a point at which this minimum is attained. Then recursively define 
$$\sigma_i = \min\{N_\cB(x)~:~x \in \cL \text{ and } x \text{ is } \F_q[T] \text{ independent of } x^{(1)},...,x^{(i-1)}\} $$
and $x^{(i)}$ a point at which this minimum is attained. 
We call $\sigma_1,...,\sigma_d$ the successive minima of $\cL$ with respect to $\cB$, and $x^{(1)},...,x^{(d)}$ their corresponding vectors.

If we let $\ell$ be the largest positive integer such that $\sigma_i \leq 1$ for $i \leq \ell$, then the definition of successive minima almost immediately implies 
$$\cL \cap \cB \subseteq \text{span}_{\F_q[T]}\{x^{(1)},...,x^{(\ell)}\}. $$

The following is proven by Mahler in \cite[equations (24) and (25)]{M1941}.

\begin{lem}\label{lem:Mahler}
Let $\cB$ be a convex body. If $\sigma_1,...,\sigma_d$ are the successive minima of $\F_q[T]^d$ with respect to $\cB$ then $x^{(1)},...,x^{(d)}$ form an $\F_q[T]$-basis for $\F_q[T]^d$ and 
$$\sigma_1...\sigma_d = \frac{1}{\vol \cB}. $$
\end{lem}
\begin{rem}
Although Mahler initially defines volume differently, he shows that this this notion of using determinants is equivalent \cite[equation (21)]{M1941}. 
\end{rem}

It is a simple corollary that this idea holds more generally (for any lattice $\cL$ not just for the integral lattices $\F_q[T]^{d}$). In~\cite{A1967}, the author states Mahler's results in this more general form, but we have not been able to find any details in the literature, so we provide them for completeness.

\begin{cor}\label{cor:Mahler_general}
Let $\cL$ be a lattice and $\cB$ a convex body. If $\sigma_1,...,\sigma_d$ are the successive minima of $\cL$ with respect to $\cB$ then $x^{(1)},...,x^{(d)}$ form an $\F_q[T]$-basis for $\cL$ and 
$$\sigma_1...\sigma_d = \frac{\det \cL}{\vol \cB}. $$
\end{cor}
\begin{proof}
As before we write $\cB = U_\cB B_1^d$ and $\cL = A_\cL \F_q[T]^d$. Note that if $\sigma'_1,...,\sigma'_d$ are the successive minima of $A_{\cL}^{-1}\cB$ with respect to $\F_q[T]^d$ (with corresponding vectors $x'^{(1)},...,x'^{(d)}$) then Lemma \ref{lem:Mahler} implies
$$\sigma'_1...\sigma'_d = \frac{1}{|\det A_\cL^{-1}U_\cB|} = \frac{\det \cL}{\vol \cB}. $$
It thus suffices to show that for each $i$,
$$\sigma_i = \sigma_i'$$
and that $x'^{(i)}$ can each be chosen such that 
$$ x'^{(i)} = A^{-1}_\cL x^{(i)}. $$
This second part ensures that $x^{(1)},...,x^{(d)}$ form a basis for $\cL$. 

Firstly, by definition we have 
\begin{align*}
    \sigma'_1
    &= \min_{\substack{x \neq 0 \\ x \in \F_q[T]^d}}\|U_\cB^{-1}A_\cL x\|\\
    &= \min_{\substack{x \neq 0 \\ A_\cL^{-1}x \in \F_q[T]^d}}\|U_\cB^{-1}x\|\\
    &= \min_{\substack{x \neq 0 \\ x \in A_\cL \F_{q}[T]^d}}\|U_\cB^{-1}x\|
    = \sigma_1.
\end{align*}
Also by construction 
$\|U_\cB^{-1}x^{(1)}\| = \sigma_1 $
which now implies 
$$\|U_\cB^{-1}A_\cL (A_\cL^{-1}x^{(1)})\| = \sigma'_1 $$ 
so we can choose $x'^{(1)} = A^{-1}_\cL x^{(1)}$. 

We now obtain for $i > 1$ that
\begin{align*}
    &\sigma_i\\
    &= \min\{\|U_\cB^{-1}x\|:x \in A_\cL\F_q[T]^d \\
    &\quad\quad\quad\quad  \text{ and } x \text{ is } \F_q[T] \text{ independent of } A_\cL x'^{(1)},...,A_\cL x'^{(i-1)}\} \\
  &= \min\{\|U_\cB^{-1}A_\cL x\|:x \in \F_q[T]^d \\
    &\quad\quad\quad\quad  \text{ and } x \text{ is } \F_q[T] \text{ independent of }  x'^{(1)},..., x'^{(i-1)}\}\\
    &= \sigma_i'
\end{align*}
which shows that $\sigma_i = \sigma_i'$. Again this now implies 
$$\|U_\cB^{-1}A_\cL (A_\cL^{-1}x^{(i)})\| = \sigma'_i $$ 
so we can choose $x'^{(i)} = A^{-1}_\cL x^{(i)}$.
\end{proof}

We will now introduce some background for \textit{dual lattices}. We firstly define a bilinear form (the dot product) $\langle ~,~\rangle $ on $\Ki^d$ by 
$$\langle(x_1,...,x_d),(y_1,..., y_d)\rangle = \sum_{i=1}^dx_iy_i. $$
Given a lattice $\cL$ we define its dual lattice 
$$\cL^* = \{x \in \Ki^d : \langle x, y \rangle \in \F_q[T] \text{ for all } y \in \cL\}. $$
Note that if $\cL = A_\cL\F_q[T]^d$ then $\cL^* = A^{-T}_\cL\F_q[T]^d$, since  $x \in \cL^*$ if and only if $A_\cL^Tx \in \F_q[T]^d$. So $\cL^*$ is in fact a lattice. 
%simple proof, go through each direction of the iff
Given a convex body $B$ we define its dual body 
$$\cB^* = \{x \in \Ki^d : |\langle x, y \rangle| \leq 1 \text{ for all } y \in \cB\}. $$
Note that again, if $\cB = U_\cB B_1^d$ then $\cB^* = U_\cB^{-T}B_1^d$ so the dual body is in fact a convex body. 

In the proof of Corollary \ref{cor:Mahler_general} we showed that the successive minima of $\cL$ with respect to $\cB$ are the same as the successive minima of $\F_q[T]^d$ with respect to $A_{\cL}^{-1}\cB$. Thus the successive minima of $\cL^* = A^{-T}_\cL\F_q[T]^d$ with respect to $\cB^*$ are equal to the successive minima of $\F_{q}[T]^d$ with respect to $A_{\cL}^TB^* = (A_\cL^{-1}\cB)^*$. Although the following result due to Mahler is only stated in \cite[equation (28)]{M1941} in the case $\cL = \F_q[T]^d$, the above discussion shows it holds in the form presented below. Perhaps the only point in \cite{M1941} that may require clarification with regards to this result is the following: given a convex body $\cB = U_\cB B_1^d$, Mahler defines \cite[page 503]{M1941} the dual to be $\cB^* = U_{\cB}^*B_1^d$ where he calls $U_{\cB}^*$ the ``complementary matrix" to $U_\cB$. His defining property of the ``complementary matrix'' on page 502 is that of the inverse transpose, and is thus consistent with our definition. 
\begin{lem}\label{lem:Mahler_dual}
Let $\cL$ be a lattice and $\cB$ a convex body. If $\sigma_1,...,\sigma_d$ are the successive minima of $\F_q[T]^d$ with respect to $\cB$ and $\sigma_1^*,...,\sigma_d^*$ are the successive minima of $\cL^*$ with respect to $\cB^*$ then 
$$\sigma_i\sigma^*_{d-i+1} = 1, ~1 \leq i \leq d. $$
\end{lem}

\subsection{Modular lattices }
We will most often apply the above results to modular lattices; that is, lattices defined by certain equations in $\F_q[T]/\langle F \rangle$. Our next step is to establish their basic properties. 
\begin{lem}\label{lem:det_of_modular_lattice}
Let $a_1,...,a_d \in \F_q[T]$ with $a_1$ coprime to $F \in \F_{q}[T]$, and let 
$$\cL = \{(x_1,...,x_d) \in \F_q[T]^d~:~ a_1x_1 + ... + a_dx_d \equiv 0 \Mod{F}\}. $$
Then $\cL$ is a lattice, and $\det \cL = |F|$. 
\end{lem}
\begin{proof}
It suffices to show 
\begin{align}\label{eq:matrix_of_modular_lattice}
    \cL = \begin{pmatrix}F & -\ov{a}_1a_2 & -\ov{a}_1a_3 & ...& -\ov{a}_1a_{d-1} & -\ov{a}_1a_d \\ 0 & 1 & 0 & ... & 0 & 0 \\ 0 & 0 & 1 & ... & 0 & 0 \\ \vdots & \vdots & \vdots & \ddots & \vdots & \vdots \\ 0 & 0 & 0 & ... &1 & 0 \\ 0 & 0 & 0 & ... & 0 & 1\end{pmatrix}\F_q[T]^d
\end{align}
where $\ov{a}_1 \in \F_q[T]$ is chosen such that $a_1\ov{a}_1 \equiv 1 \Mod{F}$. If $(x_1,...,x_d) \in \cL$ then  of course $x_1 = -\ov{a}_1a_2x_2 - \ov{a}_1a_3x_3 - ... - \ov{a}_1a_dx_d + kF$ for some $k \in \F_q[T]$. So we can write 
\begin{align*}
    (x_1,...,x_d)
    &= k(F,0,...,0) + x_2(-\ov{a}_1a_2, 1, ..., 0) + ... + x_d(-\ov{a}_1a_d, 0, ..., 1) 
\end{align*}
which is an element of the right hand side of (\ref{eq:matrix_of_modular_lattice}). 

Conversely if $(x_1, ..., x_d)$ is an element of the right hand side of (\ref{eq:matrix_of_modular_lattice}) then there exists $f_1,...,f_d \in \F_q[T]$ such that 
\begin{align*}
    x_1 &= f_1F - f_2\ov{a}_1a_2 -...-f_d\ov{a}_1a_d,\\
    x_2 &= f_2,\\
    &\vdots \\
    x_d &= f_d.
\end{align*}
This implies
\begin{align*}
    a_1x_1 + ... &+a_dx_d\\
    &= a_1f_1F - a_1f_2\ov{a}_1a_2 -...-a_1f_d\ov{a}_1a_d + a_2f_2 + ... + a_df_d
\end{align*}
and we see that since $a_1\ov{a}_1 \equiv 1 \Mod{F}$ we have 
$$a_1x_1 + ... + a_dx_d \equiv 0 \Mod{F}.$$
So $(x_1,...,x_d) \in \cL$. 
\end{proof}

We next note the following, analogous to \cite[Lemma 3.4]{BK2020}.
\begin{lem}\label{lem:dual_modular_eq}
Let $a_1,...,a_d \in \F_q[T]$ be coprime to $F$, and let 
$$\cL = \{(x_1,...,x_d) \in \F_q[T]^d~:~a_1x_1+...+a_dx_d \equiv 0 \Mod{F}\}. $$
Then 
\begin{align*}
    \cL^* = \bigg{\{}\bigg{(}&\frac{y_1}{F},...,\frac{y_d}{F}  \bigg{)} \in \F_q(T)^d ~:~ y_i \in \F_q[T]^d \text{ and } \\
     &\text{there exists some } w \in \F_q[T]^d \text{ such that } a_iw \equiv y_i \Mod{F}        \bigg{\}}.
\end{align*}
\end{lem}

\begin{proof}
We let 
\begin{align*}
    \Gamma = \bigg{\{}\bigg{(}&\frac{y_1}{F},...,\frac{y_d}{F}  \bigg{)} \in \F_q(T)^d ~:~ y_i \in \F_q[T]^d \text{ and } \\
     &\text{there exists } w \in \F_q[T]^d \text{ such that } a_iw \equiv y_i \Mod{F}        \bigg{\}} 
\end{align*}
and of course aim to show $\Gamma = \cL^*$.

Firstly, let 
$$ y = \bigg{(}\frac{y_1}{F},...,\frac{y_d}{F}  \bigg{)} \in \Gamma. $$
Then for any $x = (x_1,...,x_d) \in \cL$, by the definitions of $\Gamma$ and $\cL$ there exists $w, k_i \in \F_q[T]$ such that
$$\langle x,y \rangle = \sum_{i=1}^d\frac{x_iy_i}{F} = \sum_{i=1}^d\frac{Fk_ix_i + a_iwx_i}{F} =  \sum_{i=1}^dk_ix_i + \frac{w}{F}\sum_{i=1}^d{a_ix_i} \in \F_{q}[T].$$
Thus $\Gamma \subseteq \cL^*$.

Next, note that if $e_i$ denotes the $i$-th standard basis vector for $\Ki^d$ then $Fe_i \in \cL$. Thus for any $x = (x_1,...,x_d) \in \cL^*$ we must have 
$$\langle x, Fe_i \rangle = x_iF \in \F_q[T] $$
which means we can write $x_i = \frac{g_i}{F}$ for some $g_i \in \F_q[T]$. So we can write any $x \in \cL^*$ as 
$$x = \bigg{(}\frac{g_1}{F},...,\frac{g_d}{F}\bigg{)} $$
for some $g_i \in \F_q[T]$. 
If we now let $y_i$ be the vector with $a_i$ in the first coordinate and $-a_1$ in the $i$-th coordinate (and zeroes elsewhere) then $y_i \in \cL$ so taking $x$ as above this forces 
$$\langle x,y_i \rangle  = \frac{a_ig_1 - a_1g_i}{F} \in \F_q[T]$$
which means 
$$wa_i \equiv g_i \Mod{F} $$
where $w \equiv g_1a_1^{-1} \Mod{F}$, which of course does not depend on $i$. 

\end{proof}

\section{Counting and local-to-global techniques}\label{sec:Counting_and_localtoglobal}

In this section, we will build upon Mahler's ideas as presented in Section \ref{sec:lattices} to develop a series of new results that are useful for bounding intersections of lattices and convex bodies. These ideas will be crucial for proving all of our main results.

We first show that the intersection of a convex body and a lattice can be written precisely in terms of successive minima. We have taken inspiration from \cite{H2002}, but with some additional work we can be more exact in this setting. But before doing so, we need the following lemma. 

\begin{lem}
Suppose that
$$\cL = \text{span}_{\F_q[T]}\{v_1,...,v_d\}$$ for some basis $v_i \in \Ki$ and  let 
$$\Tilde{\cL} = \text{span}_{\F_q[T]}\{x_1v_1,...,x_dv_d\}$$
for some $x_i \in \F_q[T]$. Then 
$$|\cB \cap \cL| \leq |\cB \cap \Tilde{\cL}|\prod_{i=1}^d|x_i|. $$
\end{lem}
\begin{proof}
Let $m = |\cB \cap \Tilde{\cL}|$, and suppose there exist at least $m+1$ different points $a_1,...,a_{m+1} \in \cB\cap \cL$ such that $a_i \equiv a_1 \Mod{\Tilde{\cL}}$ for each $1 \leq i \leq m+1$. Then we have 
$$a_1 - a_i \in \cB\cap \Tilde{\cL} $$
which contradicts that $|\cB \cap \Tilde{\cL}| = m$. Thus each residue class of $\cL$ modulo $\Tilde{\cL}$ contains at most $m$ points in $\cL \cap \cB$. The result is proved, after noting that the number of residue classes modulo $\Tilde{\cL}$ in $\cL$ is exactly $\prod|x_i|$.
\end{proof}

This now leads to the following result. We note that this is essentially given in \cite[Lemma B.5]{Lee2011} in the case that $\cB = B_m^d$, and one could potentially use \cite{Lee2011} to simplify the following proof by using the ideas in the proof of Corollary \ref{cor:Mahler_general}. 

\begin{lem}\label{lem:lattice_body_intersection}
If $\sigma_1,...,\sigma_d$ are the successive minima of $\cL$ with respect to $B$ then 
$$|\cB \cap \cL| = \prod_{i=1}^d \bigg{\lceil} \frac{q}{\sigma_i} \bigg{\rceil}. $$
\end{lem}
\begin{proof}
First we will prove that 
$$|\cB \cap \cL| \geq \prod_{i=1}^d \bigg{\lceil} \frac{q}{\sigma_i} \bigg{\rceil}.  $$

Let $t$ satisfy  $\sigma_1,...,\sigma_t \leq 1$ and $\sigma_{t+1} > 1$. We note that 
$$\prod_{i=1}^d \bigg{\lceil} \frac{q}{\sigma_i} \bigg{\rceil} = \prod_{i=1}^t\frac{q}{\sigma_i} $$
since any $x \in \cL \cap \cB$ can be written as 
$$x = a_1x^{(1)} + ... +a_tx^{(t)} $$
for some $a_i \in \F_q[T]$. Since $\cL \cap \cB$ is closed under addition, we could certainly undercount the size of $\cL \cap \cB$, by counting all $a_i$ such that 
$$a_ix^{(i)} \in \cB. $$
By construction, and (\ref{eq:normfunction_inf}), we have 
$$\sigma_i = \min_{\substack{\xi \in \Ki \\ \xi x^{(i)} \in \cB}}|\xi|^{-1}. $$
Let $\xi_i$ be the point at which this minimum is achieved. Of course, if $a_ix^{(i)} \in \cB$ then $|a_i| \leq |\xi_i|$. Conversely, suppose $|a_i| \leq |\xi_i|$. Then there exists some $b_i \in B_1$ such that $a_i = \xi_ib_i$. Now since $\cB$ is a $B_1$-module and $\xi_i x^{(i)} \in \cB$ this implies that $a_ix^{(i)} = \xi_ib_ix^{(i)} \in \cB$. 
Thus, we may count all $a_i$ such that $|a_i| \leq |\xi_i| = 1/\sigma_i$. Since we are only counting $a_i \in \F_q[T]$, there are $q/\sigma_i$ such $a_i$, as desired.

Finally we show that 
$$|\cB \cap \cL| \leq \prod_{i=1}^d \bigg{\lceil} \frac{q}{\sigma_i} \bigg{\rceil}.  $$
For each $1 \leq i \leq d$ let 
$$m_i = \log_q\big{\lceil}q/\sigma_i \big{\rceil} $$
and let 
$$\Tilde{\cL} = \text{span}_{\F_q[T]}\{T^{m_1}x^{(1)},...,T^{m_d}x^{(d)}\}. $$
This means, by the previous Lemma, that 
$$|\cB \cap \cL| \leq |\cB \cap \Tilde{\cL}|\prod_{i=1}^d|T^{m_i}| = |\cB \cap \Tilde{\cL}|\prod_{i=1}^d \bigg{\lceil} \frac{q}{\sigma_i} \bigg{\rceil}.$$
Now it suffices to show that $|\cB \cap \Tilde{\cL}| = 1$. For the sake of contradiction, suppose $a \in |\cB \cap \Tilde{\cL}| \setminus \{0\}$. This means we can write
$$a = a_1T^{m_1}x^{(1)} + ... + a_kT^{m_k}x^{(k)}.$$
for some $a_i \in \F_q[T]$ where $a_k \neq 0$. This also means 
$$T^{-m_k}a = a_1T^{m_1-m_k}x^{(1)} + ... + a_kx^{(k)} \in \cL, $$
since $m_1 \geq m_2 \geq ... \geq m_k$ from the definition of successive minima. 
We can also note that $T^{-m_k}a$ is $\F_q[T]$-independent of $x^{(1)},...,x^{(k-1)}$, since $a_k\neq 0$.  Hence
\begin{align*}
    \sigma_k
     &=\min\{N_\cB(x): x \in \cL, ~x \text{ independent of } x^{(1)}, ..., x^{(k-1)}\}\\
     &\leq N_\cB(T^{-m_k}a)\\
     &= \min_{\substack{\xi \neq 0 \\ \xi T^{-m_k}a \in \cB}}|\xi|^{-1}
     \leq |T^{m_k}|^{-1}
     \leq \sigma_k/q
\end{align*}
which is a contradiction. Thus $|\cB \cap \Tilde{\cL}| = 1$, as desired. 
\end{proof}

The following is now a simple consequence of Corollary \ref{cor:Mahler_general} and Lemma \ref{lem:lattice_body_intersection}.

\begin{lemma}
\label{lem:largest} 
In general, 
$$|\cL\cap \cB|\geq q^d\frac{\vol\cB}{\det\cL}.$$
and if $\sigma_d\le 1$ then this is an equality. 
\end{lemma}

We also have the following with regards to dual lattices. 
\begin{lemma}
\label{lem:dualcount} 
Let $t$ be the largest integer such that 
$\sigma_{t}\le 1.$ Then 
$$|\cL^{*}\cap \cB^{*}|= \frac{\det\cL}{q^{2t-d}\vol\cB}|\cL\cap \cB|.$$
\end{lemma}
\begin{proof}
Note that  Lemma~\ref{lem:lattice_body_intersection} and Lemma~\ref{lem:Mahler_dual} imply 
\begin{align}
\label{eq:LLstar1}
|\cL\cap \cB|=q^{t}\prod_{j\le t}\frac{1}{\sigma_j}, \quad |\cL^{*}\cap \cB^{*}|=q^{d-t}\prod_{j\le d-t}\frac{1}{\sigma^{*}_j}.
\end{align}
Hence from Corollary~\ref{cor:Mahler_general} and Lemma~\ref{lem:Mahler_dual}
\begin{align*}
|\cL\cap \cB|=q^{t}\prod_{j\le t}\frac{1}{\sigma_j}&=q^{t}\prod_{j\le d}\frac{1}{\sigma_j}\prod_{t+1\le j \le d}\sigma_j \\
&=q^{2t-d}\frac{\vol{\cB}}{\det\cL}\prod_{1\le j \le d-t}\frac{q}{\sigma^{*}_j},
\end{align*}
and the result follows from~\eqref{eq:LLstar1}.
\end{proof}

We next have a series of general results regarding the intersections of lattices, convex bodies and arbitrary sets. By a \textit{multiset} we will mean a set with repetition. If we intersect a multiset with a set, we will again take into account repetition from the multiset. That is, if $\{x, x\}$ is a multiset then $\{x, x\} \cap \{x\} = \{x, x\}$ and $|\{x, x\}| = 2$. But of course, this will only apply if we specify that we are dealing with multisets. This flexibility will allow for greater generality with regard to the following two lemmas. 

\begin{lemma}
\label{lem:first}
Let $\cS \subseteq \F_q[T]^d \setminus \{0\}$ and $\cB\subseteq \Ki^{d}$ be a multiset and convex body, respectively. Let 
$$\cL_{\alpha}\subseteq \F_q[T]^{d}, \quad \alpha \in \cI$$ be a family of lattices such that for any distinct $\alpha, \alpha' \in I$, 

$$\cL_{\alpha} \cap \cL_{\alpha'} \cap \cB = \{0\}.$$
For each $\alpha \in \cI$ let
$$\sigma_{j,\alpha}, \quad 1\le j\le d$$
denote the successive minima of $\cL_{\alpha}$ with respect to $\cB$. If 
\begin{align}
\label{eq:only-first}
\sigma_{1,\alpha}\le 1, \quad \sigma_{2,\alpha}>1,
\end{align}
then
\begin{align*}
\sum_{\alpha \in I}|\cL_{\alpha}\cap \cS \cap \cB|^2~\leq ~ q\sum_{\substack{s_1, s_2, \in \cS \\ \lambda \in \F_q(T) \\ s_1 = \lambda s_2} }1.
\end{align*}
\end{lemma}
\begin{proof}
The assumption~\eqref{eq:only-first} implies that for each $\alpha\in \cI$ there exists some non-zero
$$a_{\alpha}\in \F_q[T]^{d},$$
such that any $s\in \cS\cap \cB\cap \cL_{\alpha}$ satisfies 
$$s=\lambda a_{\alpha},$$
for some $\lambda \in \F_q[T]\setminus \{0\}$.  For a given $\alpha$, let $\Lambda_\alpha$ be the set of all such $\lambda$ such that $\lambda a_{\alpha} \in \cL_{\alpha}\cap \cB$. Therefore 
\begin{align*}
\sum_{\alpha \in I}|\cL_{\alpha}\cap \cS \cap \cB|^2
\leq\sum_{\alpha \in I}\left(\sum_{\substack{s\in \cS}}\sum_{\substack{\lambda \in \Lambda_\alpha \\ s = \lambda a_\alpha}}1\right)^2 
&= \sum_{s_1, s_2 \in \cS}~\mathop{\sum}\limits_{\substack{\alpha \in I \\ \lambda_1, \lambda_2 \in \Lambda_\alpha^2 \\ s_1=\lambda_1a_{\alpha} \\ s_2=\lambda_2a_{\alpha}}}1.
\end{align*}
 Now for any $\lambda_1, \lambda_2$ in the above sum, let $g$ denote their greatest common divisor. Then we argue that the tuple $(s_1, s_2, \lambda_1/g, \lambda_2/g, ga_{\alpha})$ is uniquely determined by the tuple $(s_1, s_2, \lambda_1, \lambda_2, a_{\alpha})$. To see this, suppose that we had 
$$(s_1, s_2, \lambda_1/g, \lambda_2/g, ga_{\alpha}) = (s_1', s_2', \lambda_1'/g', \lambda_2'/g', g'a_{\alpha}')$$
for $(s_1, s_2, \lambda_1, \lambda_2, a_{\alpha}),(s_1', s_2', \lambda_1', \lambda_2', a_{\alpha}')$ both elements in the above sum. Then the equality $g'a_{\alpha}' = ga_{\alpha}$ implies that $a_\alpha = a_{\alpha}'$ since $g'a_{\alpha}' \in \cL_{\alpha'} \cap \cB$ and $ga_{\alpha} \in \cL_{\alpha} \cap \cB$. A simple substitution then yields $\lambda_1 = \lambda_1'$  and $\lambda_2 = \lambda_2'$. 

Thus we may write 
\begin{align*}
\sum_{\alpha \in I}|\cL_{\alpha}\cap \cS \cap \cB|^2
&\leq 
\sum_{s_1, s_2 \in \cS}~\mathop{\sum}\limits_{\substack{\lambda_1, \lambda_2 \in \F_q[T]\setminus \{0\}\\ (\lambda_1, \lambda_2) = 1 \\ a \in \F_q[T]^d \\ s_1=\lambda_1a \\ s_2=\lambda_2a}}1.
\end{align*}
Of course in the above, $(s_1, s_2, \lambda_1, \lambda_2)$ uniquely determines $a$, so we have 
\begin{align*}
\sum_{\alpha \in I}|\cL_{\alpha}\cap \cS \cap \cB|^2
&\leq 
\sum_{s_1, s_2 \in \cS}~\mathop{\sum}\limits_{\substack{\lambda_1, \lambda_2 \in \F_q[T]\setminus \{0\}\\ (\lambda_1, \lambda_2) = 1  \\ s_1=s_2 \lambda_1/\lambda_2}}1 \\
&\leq q\sum_{\substack{s_1, s_2, \in \cS \\ \lambda \in \F_q(T) \\ s_1 = \lambda s_2} }1
\end{align*}
as desired. 
\end{proof}

\begin{lemma}
\label{lem:first-last}
Let notation be as in Lemma \ref{lem:first}. If 
\begin{align}
\label{eq:only-first-1}
\sigma_{d,\alpha}\le 1,
\end{align}
and for each $\alpha\in \cI$ 
\begin{align}
\label{eq:fl-lb}
\det{\cL_{\alpha}}\ge A
\end{align}
for some non-zero $A \in \mathbb{R}$ then 
\begin{align*}
\sum_{\alpha \in I}|\cL_{\alpha}\cap \cS \cap \cB|^2 
&\le  q^d\frac{\vol{B}}{A}\sum_{\substack{s_1, s_2 \in \cS \\ s_1 = s_2}}1.
\end{align*}
\end{lemma}
\begin{proof}
Firstly by the Cauchy-Schwarz inequality 
\begin{align*}
\sum_{\alpha \in I}|\cL_{\alpha}\cap \cS \cap \cB|^2&=\sum_{\alpha \in I}\left(\sum_{x\in \cL_{\alpha}\cap \cB}\sum_{\substack{s\in \cS \\ s=x}}1\right)^2 \\ 
& \le \sum_{\alpha \in \cI}|\cL_{\alpha}\cap \cB|\sum_{x\in \cL_{\alpha}\cap \cB}|\{ s\in \cS \ : \ s=x\}|^2.
\end{align*}
The assumptions~\eqref{eq:only-first-1} and~\eqref{eq:fl-lb} combined with Corollary \ref{cor:Mahler_general} and Lemma \ref{lem:largest} imply 
\begin{align*}
\sum_{\alpha \in I}|\cL_{\alpha}\cap \cS \cap \cB|^2 
& \le q^d\frac{\vol{B}}{A}\sum_{\alpha \in \cI}\sum_{x\in \cL_{\alpha}\cap \cB}|\{ s\in \cS \ : \ s=x\}|^2\\
&\leq q^d\frac{\vol{B}}{A}\sum_{\substack{s_1, s_2 \in \cS}}\sum_{\alpha \in \cI}\sum_{\substack{x\in \cL_{\alpha}\cap \cB \\ x = s_1 \\ x=s_2}}1.
\end{align*}
Since, by assumption, any two lattices have trivial intersection inside $\cB$ we can now write 
\begin{align*}
\sum_{\alpha \in I}|\cL_{\alpha}\cap \cS \cap \cB|^2 
&\leq q^d\frac{\vol{B}}{A}\sum_{\substack{s_1, s_2 \in \cS}}\sum_{\substack{x \in \F_q[T]^d \\ x=s_1 \\ x=s_2}}1 \\
&=  q^d\frac{\vol{B}}{A}\sum_{\substack{s_1, s_2 \in \cS \\ s_1 = s_2}}1
\end{align*}
as desired. 

\end{proof}

\begin{lemma}
\label{lem:last} 
Let $\cB$ be a convex body and let $\cS,\cL\subseteq \F_q[T]^{d}$ be a set  and lattice, respectively. Let $\sigma_1,\dots,\sigma_d$ denote the successive minima of $\cL$ with respect to $\cB$. Let $t$ denote the largest integer such that
\begin{align}
\label{eq:all-last}
\sigma_{t}\le 1.
\end{align}
For any integer $k$ and  $\beta \in [1, \infty]$, we have  
\begin{align}
\label{eq:last1}
|\cL\cap \cS\cap \cB|^{k}\le \|\cS^{(k)}\|_{\beta}\left(\prod_{j\le t}\frac{q}{\sigma_j}\right)^{1-1/\beta}
\end{align}
with $\cS^{(k)}$ as in (\ref{eq:kth_convolution}). In particular, if $\sigma_d\le 1$ (that is, $t = d$) then 
\begin{align}
\label{eq:last2}
|\cL\cap \cS\cap \cB|^{k}\le q^{d(1-1/\beta)}\|\cS^{(k)}\|_{\beta}\left(\frac{\vol{\cB}}{\det{\cL}}\right)^{1-1/\beta}.
\end{align}
\begin{proof}
If $s_1,\dots,s_k\in \cL\cap \cB$ then 
$$s_1+\dots+s_k\in \cL \cap \cB.$$
Therefore 
\begin{align*}
|\cL\cap \cS\cap \cB|^{k}&\le\sum_{\substack{s_1,\dots,s_k\in \cS \\ s_1+\dots+s_k\in \cL\cap B}}1 \\ 
&=\sum_{\substack{\lambda \in \cL\cap \cB}}r(\lambda),
\end{align*}
where 
$$r(\lambda)=|\{ s_1,\dots,s_k\in \cS \ : \ s_1+\dots+s_k=\lambda\}|.$$
By H\"{o}lder's inequality 
\begin{align*}
|\cL\cap \cS\cap \cB|^{k}&\le \left( \sum_{\lambda\in \F_q[T]}r(\lambda)^{\beta}\right)^{1/\beta}|\cL\cap \cB|^{1-1/\beta} \\ 
&= \|\cS^{(k)}\|_{\beta}|\cL\cap \cB|^{1-1/\beta}.
\end{align*}
By Lemma~\ref{lem:lattice_body_intersection} and~\eqref{eq:all-last}
\begin{align*}
|\cL\cap \cS\cap B|^{k}\le \|\cS^{(k)}\|_{\beta}\left(\prod_{j\le t}\frac{q}{\sigma_j}\right)^{1-1/\beta},
\end{align*}
which establishes~\eqref{eq:last1}. The inequality~\eqref{eq:last2} follows after combining~\eqref{eq:last1} with Corollary~\ref{cor:Mahler_general}.
\end{proof}
\end{lemma}
Our next result is a consequence of Lemma~\ref{lem:last}. 
\begin{cor}
\label{cor:small-point}
With notation and conditions as in Lemma~\ref{lem:last}, suppose that $t \geq 1$. Then there exists some $x\neq 0$ satisfying
\begin{align*}
x\in \cL \cap T^\alpha\cB
\end{align*}
where
$$\alpha = \left\lfloor\log_q q\left(\frac{\|\cS^{(k)}\|_\beta}{|\cL \cap \cS \cap \cB|^k}\right)^{\beta(\beta-1)^{-1}t^{-1}}\right\rfloor.$$
\end{cor}
\begin{proof}
Under the assumption that $\sigma_t \leq 1$ Lemma \ref{lem:last} implies 
\begin{align*}
|\cL\cap \cS\cap \cB|^{k}\le \|\cS^{(k)}\|_{\beta} \left(\frac{q}{\sigma_1}\right)^{t(1-1/\beta)}.
\end{align*}
For any positive real number $A$ we let $\lfloor A \rfloor_q$ denote the largest power of $q$ less than or equal to $A$. Now since $\sigma_1$ is a power of $q$ we obtain 
\begin{align*}
    \sigma_1 \leq  \left\lfloor q\left(\frac{\|\cS^{(k)}\|_\beta}{|\cL \cap \cS \cap \cB|^k}\right)^{\beta(\beta-1)^{-1}t^{-1}}\right\rfloor_q.
\end{align*}
If $ x^{(1)} \in \cL $ is the vector corresponding to $\sigma_1$, then by definition of successive minima this means that there exists some $\xi \neq 0 $ such that $ x^{(1)} \in \xi\cB$ and 
$$|\xi| \leq  \left\lfloor q\left(\frac{\|\cS^{(k)}\|_\beta}{|\cL \cap \cS \cap \cB|^k}\right)^{\beta(\beta-1)^{-1}t^{-1}}\right\rfloor_q. $$
Since $\cB$ is closed under multiplication from $B_1$, the result follows. 
\end{proof}

\begin{lem}\label{lem:intersection_3_cases} 
Let $F \in \F_q[T]$ of degree $r$ and let $a_1, a_2, a_3$ be coprime to $F$. Let $N,M,L \geq 1$ be real numbers. Define the lattice 
$$\cL = \{(x_1,x_2,x_3) \in \F_{q}[T]^3~:~ a_1x_1 + a_2x_2 + a_3x_3 \equiv 0 \Mod{F}\}  $$
and the convex body 
$$\cB = \{(x_1, x_2, x_3) \in \F_q[T]^3~:~|x_1| \leq N,~ |x_2| \leq M,~ |x_3| \leq L\}.  $$
Let $\sigma_1, \sigma_2, \sigma_3$ denote the successive minima of $\cL$ with respect of $\cB$. Then at least one of the following holds:
\begin{enumerate}
    \item[(a)] $|\cL \cap \cB| \leq \max\{NMLq^{3-r}, 1\}, $
    \item[(b)] $\sigma_1 \leq 1$ and $\sigma_2 > 1$, or
    \item[(c)] there exists some $w$ which is coprime to $F$ and $x_1, x_2, x_3 \in \F_q[T]$ satisfying
    $$|x_1| \ll \frac{ML}{|\cL \cap \cB|},~ |x_2| \ll \frac{NL}{|\cL \cap \cB|},~ |x_3| \ll \frac{NM}{|\cL \cap \cB|}$$
    such that $a_iw \equiv x_i \Mod{F}$. 
\end{enumerate}
\end{lem}
\begin{proof}
Firstly, we may write
$$\cB  = \{(x_1, x_2, x_3) \in \F_q[T]^3~:~|x_1| \leq q^n,~ |x_2| \leq q^m,~ |x_3| \leq q^\ell\} $$
where $n = \lfloor\log_q(N)\rfloor$, $m = \lfloor\log_q(M)\rfloor$ and $\ell = \lfloor\log_q(L)\rfloor$. This will help simplify some notation throughout. 

We assume (a) does not hold. Then $|\cL \cap B| > \max\{q^3NMLq^{-r}, 1\}$ so in particular $|\cL \cap B| > 1$. Lemma \ref{lem:lattice_body_intersection} implies that $\sigma_1 \leq 1$, and we wish to show that $\sigma_3 > 1$. For the sake of contradiction if we suppose $\sigma_3 \leq 1$ then Lemma \ref{lem:lattice_body_intersection} implies 
\begin{align}\label{eq:K_bound1}
   |\cL \cap B| = \frac{q^3}{\sigma_1\sigma_2\sigma_3}.  
\end{align}
One can see that $\det \cL= q^r$ by Lemma \ref{lem:det_of_modular_lattice}, and $\vol B  = q^{m + n + \ell}$ so by Corollary \ref{cor:Mahler_general} together with (\ref{eq:K_bound1}) we have 
$$|\cL \cap B| \leq  q^{3 + m+n+\ell-r} \leq q^{3-r}NML $$
which contradicts (a) not holding. Thus if (a) does not hold we must have $\sigma_3 > 1$. If $\sigma_2 > 1$ then this is (b).

Thus, we finally consider when 
\begin{align*}
\sigma_2\le 1 \quad \text{and} \quad \sigma_3>1.
\end{align*}
By Lemma \ref{lem:lattice_body_intersection} and Corollary  \ref{cor:Mahler_general}, 
$$|\cL \cap \cB| \leq \frac{q^2}{\sigma_1\sigma_2} \leq \sigma_3q^{2 + m+n+\ell-r}.$$
By Lemma \ref{lem:Mahler_dual} this implies
$$\sigma_1^* \ll \frac{q^{m+n+\ell-r}}{|\cL \cap \cB|}.$$
The result follows from the definition of successive minima, after noting that Lemma \ref{lem:dual_modular_eq} implies
\begin{align*}
    \cL^* = \bigg{\{}\bigg{(}&\frac{x_1}{F},\frac{x_2}{F}, \frac{x_3}{F}  \bigg{)} \in \F_q(T)^d ~:~ x_i \in \F_q[T]^d \text{ and } \\
     &\text{there exists some } w \in \F_q[T]^d \text{ such that } a_iw \equiv x_i \Mod{F}        \bigg{\}}
\end{align*}
and
$$\cB^* = \{(y_1, y_2, y_3) \in \F_q[T]^3:|x_1| \leq q^{-n},~ |x_2| \leq q^{-m}, ~|x_3| \leq q^{-\ell}\}. $$
\end{proof}

We next turn our attention to a specific type of modular lattice which can correspond to polynomials with a given root modulo $F$. The following result should be considered a variant of~\cite[Lemma~15]{BGKS}.

\begin{lemma}
\label{lem:17}
Let $F\in \F_{q}[T]$ be an irreducible polynomial of degree $r$ and suppose $0<\beta \le 1$.  For positive integers $d,\ell$ with $d > 2$ and $s\in \F_{q}[T]$ suppose the lattice
$$\cL=\{ (x_0,\dots,x_{d-1})\in \F_q[T]^{d} \ : x_0+x_1s+\dots+x_{d-1}s^{d-1}\equiv 0 \Mod{F} \},$$
and convex body 
$$\cB=\{ (x_0,\dots,x_{d-1})\in \F_q(T)_{\infty}^{d} \ : \ |x_j|< q^{\ell(d-j)}\},$$
satisfy 
\begin{align}
\label{eq:lowerboundassumption}
|\cL\cap \cB|\ge q^{\ell(\beta d-1)}.
\end{align}
Additionally suppose 
\begin{align}
\label{eq:lfconds}
\ell<\min \left\{  \frac{r-3d/2}{d^2-(2\beta+1)d+5},\frac{d-2}{d-1}\frac{r}{d^2-(\beta+1)d-1} \right\}.
\end{align}
Then either:
\begin{enumerate}
\item There exists some $v \leq d-1$ and $u_0,\dots,u_{v-1}$, not all zero modulo $F$,  such that 
\begin{align}
\label{eq:8-case1}
u_0+\dots+u_{v-1}s^{v-1}\equiv 0 \mod{F}, \quad |u_j|\leq q^{\ell(v-j-(\beta d-1)/(d-2))}.
\end{align}
\item There exists $u_0,\dots,u_{d-1}$, not all zero modulo $F$, such that 
\begin{align}
\label{eq:8-case2}
u_0+\dots+u_{d-1}s^{d-1}\equiv 0 \mod{F}, \quad |u_j|\leq q^{\ell(d-j-(\beta d-1)/2) + d}.
\end{align}
\end{enumerate}
\end{lemma}
\begin{proof}
Let $\sigma_1,\dots,\sigma_d$ denote the successive minima of $\cL$ with respect to $\cB$. Define  $t \leq d$ to be the largest integer such that $\sigma_t \leq 1$. We proceed on a case-by-case basis depending on the size of $t$.

\textit{Case I: $t=d$.} Lemma~\ref{lem:det_of_modular_lattice} and a simple calculation give 
\begin{align}
\label{eq:vol}
\vol\cB=q^{\ell d(d+1)/2-d}, \quad \det{\cL}=q^{r}.
\end{align}
Then Lemma~\ref{lem:largest} and (\ref{eq:lowerboundassumption}) imply
\begin{align*}
\ell \geq \frac{2r}{d^2-(2\beta-1)d+2},
\end{align*}
contradicting our assumption~\eqref{eq:lfconds}. Thus, we cannot have $t=d$.

\textit{Case II: $t=d-1$.} By Lemma~\ref{lem:Mahler_dual}, this implies that 
\begin{align}
\label{eq:sigma1-star}
\sigma_1^{*}< 1, \quad \sigma_2^{*}\ge 1.
\end{align}
By~\eqref{eq:lowerboundassumption},~\eqref{eq:vol} and  Lemma~\ref{lem:dualcount}
\begin{align}\label{eq:poly_lat_lem_dual_bound}
|\cL^{*}\cap B^{*}|\ge {q^{r +2- \ell(d^2-(2\beta-1)d+2)/2}}.
\end{align}
In particular, by~\eqref{eq:lfconds}, this implies $|\cL^* \cap \cB^*| > 1$.
By Lemma~\ref{lem:dual_modular_eq} and the fact that 
$$B^{*}=\left\{ (x_0,\dots,x_{d-1})\in \F_q(T)_{\infty}^{d} \ : \ |x_j|\le \frac{1}{q^{\ell(d-j)-1}} \right\},$$
we see there exists $u_0,\dots,u_{d-1}\in \F_q[T]$, not all zero, satisfying 
\begin{align*}
u_0s^{j}\equiv u_j \Mod{F},~ \frac{1}{F}\left(u_0,...,u_{d-1} \right) \in \cB^*.
\end{align*}
Also, from~\eqref{eq:sigma1-star}, \eqref{eq:poly_lat_lem_dual_bound} and Lemma \ref{lem:lattice_body_intersection} we can say 
$$\sigma_1^* \leq q^{\ell(d^2-(2\beta-1)d+2 + 2d/\ell)/2-r-2}$$
Thus by definition of successive minima we can additionally say 
$$|u_j|\leq q^{r - \ell(d-j) + 1 -r-2+\ell(d^2-(2\beta-1)d+2 +2d/\ell)/2 +1 } = q^{\ell((d^2-(2\beta+1)d+2)/2+ j) + d}.$$
Keeping in mind (\ref{eq:lfconds}), we can apply Lemma \ref{lem:recursive}, so there exists $a$ and $b$, both non-zero modulo $F$, satisfying $bs\equiv a \Mod{F}$ and
 \begin{align*}
\quad |a|\leq q^{\ell d(d+1-2\beta)/(2d-2) + d/(d-1)}, \quad |b|\leq q^{\ell(d^2-(2\beta+1)d+2)/(2d-2)+d/(d-1)}.
\end{align*}
and this is stronger than~\eqref{eq:8-case2}.

\textit{Case III: $2\le t \leq d-2$}.  Of course this case is only possible if $d \geq 4$, but we will take care of the case $t=1$ later. By Corollary~\ref{cor:small-point} with $\beta=\infty,k=1$ and $\cS=\F_q[T]$, there exists some nonzero $(v_0,\dots,v_{d-1})\in \F_q[T]^{d}$  satisfying 
\begin{align*}
v_0 + v_1s + ... + v_{d-1}s^{d-1} \equiv 0 \Mod{F}, ~|v_j| \le q^{\ell(d-j-(\beta d-1)/(d-2))}.
\end{align*}
If we let $P' \in \F_q[T][x]$ denote the polynomial 
$$P'(x) = v_0 + ... + v_{d-1}x^{d-1} $$
then $P'$ must have an irreducible factor 
$$P(x)=u_0+\dots+u_{v-1}x^{v-1},$$
with $v\le d$ such that 
\begin{align*}
P(s)\equiv 0 \Mod{F}.
\end{align*}
By Lemma \ref{lem:factorisation-bound}, $u_0,\dots,u_{v-1}$ satisfy
\begin{align*}
 |u_j|\le q^{\ell(v-j-(\beta d-1)/(d-2))}.
\end{align*}
Let $x^{(2)},...,x^{(t)}$ denote the vectors corresponding to $\sigma_2,...,$ $\sigma_{t}$ respectively. These are of course linearly independent by definition, and to $x^{(j)}$ we can associate a polynomial 
$$Q_j(x)=u_{j,0}+\dots+u_{j,d-1}x^{d-1}, \quad |u_{j,i}|< q^{\ell(d-i)},$$
such that $Q_{j}(s)\equiv 0 \mod{F}$.
We will apply Corollary~\ref{cor:good} with $a = -(\beta d-1)/(d-2),~b=0$ and $e=d$, since $\deg{P},\deg{Q_1},\dots,\deg{Q_{t}}\le d-1$. 

Note the condition~\eqref{eq:lfconds} implies~\eqref{eq:deabconds} is satisfied. 
Hence by Corollary~\ref{cor:good}, we see that
\begin{align*}
v\le d-(t-1),
\end{align*}
which implies~\eqref{eq:8-case1}.

\textit{Case IV: $t = 1$}. 
By~\eqref{eq:lowerboundassumption} and Lemma \ref{lem:lattice_body_intersection}, 
\begin{align*}
\sigma_1\leq q^{d-\ell(\beta d-1)},
\end{align*}
and \eqref{eq:8-case2} is immediate
from the definition of successive minima. 
\end{proof}
We note that the special case of Lemma~\ref{lem:17} with $\beta=1$ does not produce as strong a result as the direct function field analogue of~\cite[Lemma~15]{BGKS}. This is due to our inability to optimise between cases (I)--(IV) without a specific range of values for $\beta$. It is possible to refine the proof of Lemma~\ref{lem:17} to produce a direct comparison to~\cite[Lemma~17]{BGKS} in the case $\beta=1$.

\section{Points on curves  }

\subsection{Proof of Theorem~\ref{thm:curveinbox_1} }
After applying a change of variables, we may assume that $s_1 = s_2 = 0$. Now we define the set 
$$\mathcal{S} = \{ (x,y) \in \cI_m^2 : \Phi(x,y) \equiv 0 \Mod{F}\},$$
 lattice
$$\cL = \{(y, x_1, ..., x_d) \in \F_q[T]^{d+1} : y - a_1x_1 - ... - a_dx_d \equiv 0 \Mod{F}\} $$
and convex body 
$$\cB = \{(y, x_1, ..., x_d) \in \Ki^{d+1} : |y| \leq q^{m} \text{ and } |x_i| \leq q^{im}\}. $$
Of course, we may suppose that $\cS$ is non-empty. Thus, fix $(x_0, y_0) \in \cS$ and let 
\begin{align*}
\cS_0=\{(y-y_0,x-x_0,x^2-x_0^2,\dots,x^{d}-x_0^d) \ : \ (x,y)\in \cS \},
\end{align*}
and note that 
\begin{align}
\label{eq:curves1}
|\cS|\leq|\cS_0\cap \cL \cap \cB|.
\end{align}
Let $\sigma_1, ..., \sigma_{d+1}$ denote the successive minima of $\cL$ with respect to $\cB$. We consider two cases depending on if $\sigma_{d+1} > 1$ or not.

Case I: $\sigma_{d+1}\le 1$.
Applying~\eqref{eq:curves1} and Lemma~\ref{lem:last} with 
\begin{align*}
\beta=\infty, \quad k=d(d+1),
\end{align*}
gives 
\begin{align*}
|\cS|^{k}\le q^{d+1}\|\cS_0^{(k)}\|_{\infty}\frac{\vol{\cB}}{\det{\cL}}.
\end{align*}
By Lemma \ref{lem:det_of_modular_lattice} $\det \cL = q^r$ and a simple computation yields $\vol \cB = q^{m(d^2+d+2)/2}$. Combined with the above, this implies 
\begin{align}
\label{eq:Ss0Ss0-1}
|\cS|^{k}\le q^{d+1}\|\cS_0^{(k)}\|_{\infty}\frac{q^{m(d^2+d+2)/2}}{q^{r}}.
\end{align}
Let $\pi(\cS)$ denote the projection of $\cS$ onto the first coordinate, so that 
\begin{align*}
|\pi(\cS)|=|\cS|.
\end{align*}
We note that there exists $g_1,...,g_d \in \F_q[T]$ such that $\|\cS_0^{(k)}\|_{\infty} $ is bounded by the number of solutions to the system of equations
$$x_1^j + ... + x_{k}^j = g_j, ~1 \leq j \leq d $$
with $x_1,...,x_{k} \in \pi(\cS)$. By Corollary \ref{cor:wooley} this implies 
$$\|\cS_0^{(k)}\|_{\infty} \leq |\mathcal{S}|^{k/2}q^{o(n)},$$
which combined with~\eqref{eq:Ss0Ss0-1} results in
\begin{align}\label{eq:CB1:S_bound2}
|\mathcal{S}|  \leq (q^{m(d^2+d+2)/2-r + o(m)})^{2/k} = q^{m + 2m/(d^2+d) - 2r/(d^2+d) + o(m)}.
\end{align}

Case II: $\sigma_{d+1}>1$.
By  Lemma \ref{lem:dual_modular_eq}
$$\cL^* = \frac{1}{F}\{(z, g_1, ..., g_{d}) \in \F_q[T]^d: -a_iz \equiv g_i \Mod{F}\} $$
and note that 
$$\cB^* = \{(z, g_1, ..., g_d): |z| \leq q^{-m} \text{ and }|g_i| \leq q^{-im}\}. $$
Let $\sigma_1^*$ denote the first successive minima of $\cL^*$ with respect to $\cB^*$. 
By Lemma \ref{lem:Mahler_dual}
$$\sigma_1^* < 1. $$
This implies that 
$$\cL^* \cap \cB^* \neq \{0\},$$
and hence there exists $z, g_1, ..., g_d \in \F_q[T]$  such that $-a_iz \equiv g_i \Mod{F}$ and 
\begin{align}\label{eq:CB1:z_gi_bounds}
    |z| \leq q^{r-m}, ~|g_i| \leq q^{r-im}.
\end{align}
Note that we have $z \neq 0$ and $\gcd(z, F) = 1$. If $(x,y) \in \mathcal{S}$ then by construction
$$zy \equiv -g_0 - g_1x - ... - g_dx^d \Mod{F} $$
where $|g_0| < q^r$ satisfies $g_0 \equiv -a_0z \Mod{F}$. 
Thus, it suffices to count solutions to 
$$zy + g_0 + g_1x + ... + g_dx^d = tF $$
for $t \in \F_q[T]$. Since $\gcd(a_d, F) = 1$ this implies $g_d \neq 0$. We have by (\ref{eq:CB1:z_gi_bounds}) and $(x,y) \in \cI_m^2$ that 
$$|zy + g_0 + g_1x + ... + g_dx^d| < q^r $$
and thus $t = 0$. Applying Lemma \ref{lem:Bombieri_pila} now implies 
\begin{align}\label{eq:CB1:S_bound1}
    |\mathcal{S}| \leq q^{m/d + o(n)}.  
\end{align}

Combining (\ref{eq:CB1:S_bound1}) and (\ref{eq:CB1:S_bound2}) yields the result.

\subsection{Proof of Theorem~\ref{thm:curveinbox_2}  }
Before presenting the proof of Theorem~\ref{thm:curveinbox_2}, we need a lemma which may be of independent interest. This construction is very standard, and should be considered a function field analogue of the main construction in Coppersmith's Theorem (see \cite{Coppersmith2001}). 
\begin{lem}\label{lem:coppersmith}
Let $Q(x) \in \F_q[T][x]$ of degree $d$ with leading coefficient coprime to $F$. For any $\epsilon > 0$, if 
\begin{align}\label{eq:coppersmith_condition}
    m < r\left(\frac{1}{d}- \epsilon\right)
\end{align}
then number of solutions to 
$$Q(x) \equiv 0 \Mod{F}, ~\deg x < m$$
is $O(1)$, where the implied constant depends on at most $q,d$ and $\epsilon$.
\end{lem}

\begin{proof}
Since the leading coefficient of $Q$ is coprime with $F$ and we are only interested in counting roots of $Q$ modulo $F$, we may assume that $Q(x)$ is monic. Let $h$ be an integer to be specified later, and consider the lattice $\cL \subseteq \F_q[T]^{d(h+1)}$ generated by the matrix whose columns consist of the coefficients of the polynomials 
\begin{align}\label{eq:coppersmith_polynomials}
   F^{h-v}Q(x)^vx^u, ~u \in \{0,...,d-1\}, ~v \in \{0,...,h\}. 
\end{align}
These can be arranged in a matrix to be upper triangular, which makes computing $\det \cL$ simple and we find 
$$\det \cL  = q^{rdh(h+1)/2}. $$
Additionally for any $(y_0,...,y_{d(h+1)-1}) \in \cL$, if we let 
\begin{align}\label{coppersmith_g}
    g(x) = y_0 + y_1x + ... + y_{d(h+1)-1}x^{d(h+1)-1}
\end{align}
then 
\begin{align}\label{eq:coppersmith_modF^h}
    g(x_0) \equiv 0 \Mod{F^h}
\end{align}
for any solution $x_0$ to $Q(x_0) \equiv 0 \Mod{F}$, since $g(x)$ is an $\F_q[T]$ combination of the polynomials \eqref{eq:coppersmith_polynomials}, and $F^h$ divides $F^{h-v}Q(x_0)^v$. 

Next, let 
$$\cB = \{(x_0, ..., x_{d(h+1)-1}) ~:~ |x_i| < q^{rh-im}\}.$$
Another computation yields 
$$\vol \cB = q^{dh^2+dhr-dhm(dh+d-1)/2 - dm(dh+d-1)/2 - d(h+1)}.$$
By Lemma \ref{lem:largest}, we know that $\cL \cap \cB$ will contain a non-zero point if 
$$q^{d(h+1)}\frac{\vol \cB}{\det \cL} > 1.$$
Simplifying all of the details down, this is equivalent to
\begin{align}\label{eq:coppersmith_non-zero}
    m < r\frac{h}{hd + d - 1}.
\end{align}
If we choose any
$$h > \frac{1-d\epsilon}{d\epsilon} $$
then by \eqref{eq:coppersmith_condition}, we have that \eqref{eq:coppersmith_non-zero} is satisfied. 
Thus, let $g(x)$ be as in \eqref{coppersmith_g} for some non-zero $(y_0,..., y_{d(h+1)-1}) \in \cL \cap \cB$. Then as stated in \eqref{eq:coppersmith_modF^h}, $g(x_0) \equiv 0 \Mod{F^h}$ for any solution $x_0$ to $Q(x_0) \equiv 0 \Mod{F}$. The definition of $\cB$ implies that in fact $g(x_0) = 0$, and thus there are at most $qd(h+1)$ such $x_0$. 
\end{proof}
We can now proceed with the proof of Theorem \ref{thm:curveinbox_2}. Suppose we have a solution $(x_0, y_0)$ to the congruence 
$$\Phi(x,y) \equiv 0 \Mod{F},~ (x,y) \in \cI_m(s_1, s_2). $$
Then the change of variables $(x,y) \mapsto (x_0-x, y-y_0)$ yields 
$$y^2-y(2y_0) \equiv x^3(-a_3) + x^2(3a_3x_0 + a_2) + x(-3a_3x_0^2 -2a_2x_0 - a_1) \Mod{F}. $$
We note that $(x_0-x, y-y_0) \in \cI_m^2$. Thus, it suffices to estimate the size of the set 
\begin{align*}
    \cS = \{(x,y) \in \cI_m^2 ~:~y^2+ c_1y \equiv b_3x^3 + b_2x^2 + b_1x \Mod{F}\}
\end{align*}
where $c_1, b_1, b_2, b_3 \in \F_q[T]$ with $\gcd(b_3, F) = 1$. 
We define the lattice,
\begin{align*}
    \cL = \{(x_1, x_2, x_3, &y_1, y_2) \in \F_q[T]^5 : \\
            &b_1x_1 + b_2x_2 + b_3x_3 - c_1y_1 - y_2 \equiv 0 \Mod{F}\}
\end{align*}    
and the convex body 
\begin{align*}
    \cB = \{(x_1, x_2, x_3, &y_1, y_2) \in \Ki^5 : \\
            &|x_1|, |y_1| \leq q^m, ~|x_2|, |y_2| \leq q^{2m}, ~|x_3| \leq q^{3m}\}. 
\end{align*}
Let $\sigma_1, ..., \sigma_5$ denote the successive minima of $\cL$ with respect to $\cB$. We will break this discussion into two cases. 

\textit{Case I: $\sigma_5 \leq 1$.} We let $\pi(\cS)$ denote the projection of $\cS$ onto the first coordinate, so that
$$|\cS| \ll |\pi(\cS)| $$
by Lemma \ref{lem:coppersmith}. 
If we define
$$\cS_0 = \{(x_1+...+x_6, x_1^2+...+x_6^2,x_1^3+...+x_6^3) ~:~x_1,...,x_6 \in \pi(\cS)\}$$
we have 
\begin{align}\label{eq:PC2_j_1}
    \sum_{s \in \cS_0}J_{s,3,3}(\pi(\cS))  = |\pi(\cS)|^6
\end{align}
with $J_{s,3,3}(\pi(\cS))$ as in (\ref{eq:J_system}). We also note that 
$$\sum_{s \in \cS_0}J_{s,3,3}(\pi(\cS))^2 $$
is bounded by the number of solutions to 
$$x_1^j + ... + x_6^j = x_7^j + ... + x_{12}^j, ~1 \leq j \leq 3, ~x_j \in \pi(\cS). $$
Thus by Lemma \ref{lem:wooley_VMVT}, 
\begin{align}\label{eq:PC2_j_2}
    \sum_{s \in \cS_0}J_{s,3,3}(\pi(\cS))^2 \leq q^{o(m)}|\pi(\cS)|^6.
\end{align}
Now applying the Cauchy-Schwarz inequality with (\ref{eq:PC2_j_1}) and (\ref{eq:PC2_j_2}) yields 
\begin{align}\label{eq:PC2_S^6_1}
   |\cS|^6 \leq q^{o(m)}|\cS_0|. 
\end{align}
For any 
$$s = (x_1+...x_6, x_1^1+...+x_6^2, x_1^3+...+x_6^3) \in \cS_0 $$
since each $x_i $ corresponds to at least one $(x_i, y_i) \in \cS$ we have that 
$$(x_1+...x_6, x_1^1+...+x_6^2, x_1^3+...+x_6^3, y_1+...+y_6, y_1^2+...+y_6^2) \in \cL \cap \cB. $$
Of course this point in $\cL \cap \cB$ is uniquely determined by $s$ so by (\ref{eq:PC2_S^6_1}), 
\begin{align*}
   |\cS|^6 \leq q^{o(m)}|\cL \cap \cB|. 
\end{align*}
Lemma \ref{lem:det_of_modular_lattice} gives $\det \cL = q^r$ and it is easy to see $\vol \cB = q^{9m}$. Then since $\sigma_5 \leq 1$, by Lemma \ref{lem:largest} we conclude 
\begin{align*}
   |\cS|^6 \ll q^{9m-r + o(m)}
\end{align*}
and the result follows in this case. 

\textit{Case II: $\sigma_5 > 1$.} Lemma \ref{lem:dual_modular_eq} implies 
\begin{align*}
    \cL^* = \frac{1}{F}\{(g_1, g_2, &g_3, z_1, z_2) \in \F_q[T]^5 :\\
    &-b_iz_2 \equiv g_i \Mod{F}\text{ and } c_1z_2 \equiv z_1 \Mod{F}\}
\end{align*}
and by definition 
\begin{align*}
    \cB^* = \{(g_1, g_2, g_3, &z_1, z_2) \in \Ki^5 : \\
            &|g_1|, |z_1| \leq q^{-m}, ~|g_2|, |z_2| \leq q^{-2m}, ~|g_3| \leq q^{-3m}\}.
\end{align*}
If $\sigma_1^*$ denotes the first successive minima of $\cL$ with respect to $\cB^*$ then since $\sigma_5 >  1$, Lemma \ref{lem:Mahler_dual} implies $\sigma_1^* < 1$. This implies that $\cB^* \cap \cL^* \neq \{0\}$. So there exists non-zero $(g_1, g_2, g_3, z_1, z_2) \in \F_q[T]^5$ satisfying  
\begin{align*}
    |g_1|, |z_1| \leq q^{r-m}, ~|g_2|, |z_2| \leq q^{r-2m}, ~|g_3| \leq q^{r-3m}
\end{align*}
and 
\begin{align*}
    -b_iz_2 \equiv g_i \Mod{F}\text{ and } c_1z_2 \equiv z_1 \Mod{F}\}.
\end{align*}
This implies $z_2 \neq 0$, since otherwise $g_1=g_2=g_3=z_1=z_2=0$. Thus to count $|\cS|$, it suffices to count the number of solutions to 
$$z_2y^2 + z_1y + g_3x^3 + g_2x^2 + g_2x^2 + g_1x = tF $$
for some $t \in \F_q[T]$ and $(x,y) \in \cI_m^2$. The bounds on $g_i$ and $z_i$ imply that $t=0$. Thus applying Lemma \ref{lem:Bombieri_pila},
$$|\cS| \leq q^{m/3+o(m)}. $$

\section{Kloosterman equations }
\subsection{Proof of Theorem~\ref{thm:sums_of_inverses} }
We note that trivially we have $$E_{F,k}^\inv(\cI_m) \leq q^{m(2k-1)}$$ so we may assume $km < r$ as otherwise the result is trivial. 

For any $\lambda \in \F_q[T]$  we denote 
$$I_{F, \lambda, k}(\cI_m) = \left|\{(x_1,...,x_k) \in \cI_m^k : \ov{x_1} + ... + \ov{x_k} \equiv \lambda \Mod{F}\}\right|.  $$
Since $m < r/k$ we can say $I_{F, 0, k}(\cI_m)= 0$.  Thus if we let 
$$\Omega = \{0 \neq \lambda  \in \cI_r : I_{F, \lambda, k}(\cI_m) \geq 1\} $$
then of course 
$$E_{F,k}^\inv(\cI_m) \leq \sum_{\lambda \in \Omega}I_{F, \lambda, k}(\cI_m)^2.$$
We now define the \textit{multiset} 
$$\cS = \{(x_1...x_k, ~x_2...x_k + ... + x_1...x_{k-1}) ~:~ x_i \in \cI_m\setminus \{0\}\}.$$
We recall a {multiset} may have repetition of elements. For each $\lambda \in \Omega $ we define the lattice 
$$\cL_\lambda  = \{(x,y) \in \F_q[T]^2 : \lambda x \equiv y \Mod{F}\} $$
and convex body 
$$\cB = \{(x,y) \in \Ki^2: |x| \leq q^{mk}, |y| \leq q^{m(k-1)}\}. $$
It is clear that $\vol \cB = q^{m(2k-1)}$ and by Lemma \ref{lem:det_of_modular_lattice}, $\det \cL_{\alpha} = q^r$. Also if $\lambda_1 \neq \lambda_2$ then $\cL_{\lambda_1} \cap \cL_{\lambda_2} \cap \cB = \{0\}$ since $km < r$. We let $\sigma_{1, \lambda}, \sigma_{2, \lambda}$ denote the successive minima of $\cL$ with respect to $\cB$. 

If $(x_1,...,x_k)$ is counted by $I_{F, \lambda, k}(\cI_m)$ then 
$$\lambda x_1...x_k \equiv x_2...x_k + ... +
x_1...x_{k-1}\Mod{F} $$
which implies
$$(x_1...x_k, x_2...x_k + ... + x_1...x_{k-1}) \in \cL_\lambda \cap \cB.$$
This indicates that $\sigma_{1, \lambda} \leq 1$ for every $\lambda \in \Omega$, and it also indicates that we may write 
$$E^\inv_{F,k}(\cI_m) \leq \sum_{\substack{\lambda \in \Omega \\ \sigma_{2, \lambda} > 1}}|\cL_\lambda  \cap \cB \cap \cS|^2 + \sum_{\substack{\lambda \in \Omega \\ \sigma_{2, \lambda} \le 1}}|\cL_\lambda  \cap \cB \cap \cS|^2.$$
We recall that when we intersect a multiset with a set, we count repetition. For the first sum, we apply Lemmas \ref{lem:SOI:inverses_equality} 
 and \ref{lem:first} to obtain 
\begin{align*}
    \sum_{\substack{ \sigma_{2, \lambda} > 1}}|\cL &\cap \cB \cap \cS|^2\\
    &\leq  \left|\left\{(x_1, ..., x_{2k}) \in \cI_m^{2k} : \frac{1}{x_1}+...+\frac{1}{x_k} = \frac{1}{x_{k+1}}+...+\frac{1}{x_{2k}}\right\}\right| \\ 
&  \leq q^{km + o(m)}.
\end{align*}
For the second sum, we apply Lemma \ref{lem:first-last} to obtain
\begin{align*}
    \sum_{\substack{\lambda \in \Omega \\ \sigma_{2, \lambda} \le 1}}|\cL_\lambda  \cap \cB \cap \cS|^2
    &\ll q^{m(2k-1)-r}T
\end{align*}
where $T$ is the number of solutions to the system 
\begin{align*}
    x_1...x_k &= y_1...y_k\\
    x_2...x_{k}+...+x_{1}...x_{k-1} &= y_2...y_{k}+...+y_1...y_{k-1}
\end{align*}
with $x_i, y_i \in \cI_m$. This implies by Lemma \ref{lem:SOI:inverses_equality} that
\begin{align*}
    \sum_{\substack{\lambda \in \Omega \\ \sigma_{2, \lambda} \le 1}}|\cL_\lambda  \cap \cB \cap \cS|^2
    &\ll q^{m(3k-1)-r + o(m)}.
\end{align*}

The result follows from combining the above estimates.

\subsection{Proof of Theorem~\ref{thm:sums_of_inverses-general}}

Our argument incoporates Lemma~\ref{lem:17} into the proof of~\cite[Theorem~1.1]{SZ2018} and this is where our improvement comes from. 

Recalling~\eqref{eq:inv-def}, we see that $E^\inv_{F,k}\left( \cI_m(s) \right)$ counts the number of solutions to 
\begin{align}
\label{eq:inv-eqn-1}
\ov{s+x_1} + ... + \ov{s+x_k} \equiv \ov{s+x_{k+1}} + ... + \ov{s+x_{2k}} \Mod{F}, \quad \deg{x_j}\le m.
\end{align}
Let $E^{*}$ count the number of solutions to~\eqref{eq:inv-eqn-1} subject to the extra condition that 
\begin{align}
\label{eq:x-inv-conds}
|\{x_1,\dots,x_{2k}\}|=2k.
\end{align}
That is, solutions such that each of $x_1,...,x_{2k}$ is pairwise distinct.
In the proof of~\cite[Theorem~1.1]{SZ2018}, it is demonstrated that it is sufficient to show 
\begin{align*}
E^{*}\le q^{km+o(m)}.
\end{align*}
Thus we will assume, for the sake of contradiction, that for all sufficiently large $m$ that 
\begin{align}
\label{eq:inv-toshow-cont}
E^{*}\ge q^{km(1+\epsilon)}
\end{align}
for some $\epsilon > 0$. 
 For each tuple $x=(x_1,\dots,x_{2k})$ satisfying~\eqref{eq:inv-eqn-1} and~\eqref{eq:x-inv-conds}, consider the polynomial in $\F_q[T][Z]$
\begin{align}
\label{eq:pxpx}
P_{x}(Z)=\sum_{s=1}^{k}\prod_{j\neq s}(x_j+Z)-\sum_{s=k+1}^{2k}\prod_{j\neq s}(x_j+Z).
\end{align}
Note that by construction
$$P_{x}(s)\equiv 0 \Mod{F}$$
and the assumption~\eqref{eq:x-inv-conds} implies that 
\begin{align*}
P_x(-x_1)=\prod_{j\neq 1}(x_j-x_1)\not \equiv 0 \Mod{F}, 
\end{align*}
hence the polynomial $P_x$ is not constant modulo $F$. For $\deg x< m$ let $E^{*}(x)$ count the number of solutions to~\eqref{eq:inv-eqn-1} with variables satisfying~\eqref{eq:x-inv-conds} subject to the further restriction that 
$$x_1=x.$$
By~\eqref{eq:inv-toshow-cont} and the pigeonhole principle, there exists some $x_0$ with $\deg x_0 < m$ such that 
$$E^{*}(x_0)\ge q^{(k-1)m(1+\epsilon)}.$$
Let $\cL$ denote the lattice 
$$\cL=\{ (y_0,\dots,y_{2k-1})\in \F_q[T]^{2k} \ : \ y_0+\dots+y_{2k-1}s^{2k-1}\equiv 0 \Mod{F}\},$$
and $\cB$ the convex body
$$\cB=\{ (y_0,\dots,y_{2k-1})\in \F_q(T)_{\infty}^{2k} \ : \ |y_j| < q^{(2k-j)m}\}.$$
We next show that~\eqref{eq:inv-toshow-cont} implies 
\begin{align}
\label{eq:LB-inverses}
|\cL\cap \cB|\ge q^{(k-1)m(1+\epsilon)}.
\end{align}
Each point $x=(x_1,\dots,x_{2k})$ counted by $E^{*}(x_0)$ corresponds to a polynomial $P_x$ as in~\eqref{eq:pxpx}. It is clear that the vector formed from the coefficients of $P_x$ belongs to $\cL\cap \cB$. Hence in order to establish~\eqref{eq:LB-inverses} it is sufficient to show for any $Q\in \F_q[T][Z]$
\begin{align*}
|\{ (x_1,\dots,x_{2k})\in \cI_m^{2k} \ :  \ P_x(Z)=Q(Z), \  x_1=x_0\}|\le q^{o(m)}.
\end{align*}
If $(x_1,\dots,x_{2k})$ satisfies the above conditions, then
\begin{align*}
P_x(-x_0)=\prod_{j\neq 1}(x_j-x_0)=Q(-x_0),
\end{align*}
to which Lemma~\ref{lem:divisorbound} implies there are at most $q^{o(m)}$ solutions in variables $x_2,\dots,x_{2k}$, since for each $2\le j \le 2k$, $x_j-x_0$ is a divisor of a fixed $Q(-x_0)$ which satisfies $\deg(Q(-x_0))\ll km$.

Now that we have established ~\eqref{eq:LB-inverses}, by~\eqref{eq:kloostermanpar} and~\eqref{eq:LB-inverses},we may apply Lemma~\ref{lem:17} with $\beta=1/2$. Of course there are two cases to consider. The first case is that there exists some sequence $u_0,\dots,u_{2k-2}$, not all zero modulo $F$, satisfying 
\begin{align*}
u_0+u_1s+\dots+u_{2k-2}s^{2k-2}\equiv ~&0 \Mod{F}, \\
&|u_j|\leq q^{m(2k-1-j-((2k-1)-1)/(2k-3))}.
\end{align*}
Let $Q(Z)$ denote the polynomial
\begin{align*}
Q(Z)=u_0+u_1X+\dots+u_{2k-2}X^{2k-2}.
\end{align*}
For each tuple $x=(x_1,\dots,x_{2k})$ satisfying~\eqref{eq:inv-eqn-1} and~\eqref{eq:x-inv-conds}, the polynomial $P_{x}$ and $Q$ share a common root modulo $F$. Hence 
\begin{align*}
\Res(Q,P_x)\equiv 0 \Mod{F}.
\end{align*}
By~\eqref{eq:kloostermanpar} and applying Lemma ~\ref{lem:resultantbound} to $Q$ and $P_x$, we can conclude 
\begin{align*}
\Res(Q,P_x)=0.
\end{align*}
The second case to consider is when there exists some sequence $u_0,...,u_{2k-1}$, not all zero modulo $F$, satisfying 
\begin{align*}
u_0+u_1s+\dots+u_{2k-2}s^{2k-2}\equiv 0 \mod{F}, \quad |u_j|\leq q^{m(2k-j-(k-1)/2 + 2k/m)}
\end{align*}
and associate to it a polynomial $Q$. Arguing similarly to the above, we can conclude that 
$$\Res(Q, P_x) = 0$$

Regardless of the case, we have a polynomial $Q$ such that $\Res(Q, P_x) = 0$ for any tuple $x = (x_1, ..., x_{2k})$ satisfying~\eqref{eq:inv-eqn-1} and~\eqref{eq:x-inv-conds}. Let $\beta_1,\dots,\beta_{t}$ denote the distinct roots of $Q$ in $\overline \F_q(T)$, of course for some $t \ll k$. It follows that there exists some $1\le j \le t$ such that 
$$E^{*}\ll |\{ (x_1,...,x_{2k}) \in \cI_m^2 \ : \ P_x(\beta_j)=0 \}|.$$
Note that if $P_x(\beta_j)=0$ then 
\begin{align*}
\frac{1}{x_1+\beta_j}+\dots-\frac{1}{x_{2k}+\beta_j}=0.
\end{align*}
It follows from Lemma~\ref{lem:SOI:inverses_equality} that 
\begin{align*}
E^{*}\leq q^{km+o(m)},
\end{align*}
which completes the proof. 

\section{Modular roots }
We will introduce a lemma and some notation that will be useful for proving each of Theorems \ref{thm:energy_bound}, \ref{thm:energy_1m_improved} and \ref{thm:T4_bound}. For any $\lambda \in \cI_r\setminus\{0\}$ we define 
$$Q_{F, \lambda}(\balpha) = \sum_{\substack{(u,v) \in \cI_r^2 \\ u-v \equiv \lambda (F)}} \alpha({u^2})\ov{\alpha({v^2})}.$$
For shorthand we will right $Q_{F, \lambda, m} = Q_{F, \lambda}(\bm{1}_m)$, and of course we have 
\begin{align}\label{eq:Q_1m}
    Q_{F, \lambda, m} = |\{(u,v) \in \cI_r^2:~ \deg_F(u^2),&\deg_F(v^2) < m,\\
    &u-v \equiv \lambda\Mod{F}\}| \nonumber . 
\end{align}
We will also define 
\begin{align*}
    J_{F, \lambda, m} = \left|\{(x,y) \in \cI_m^2: x^2 + \lambda^4 \equiv 2\lambda^2y \Mod{F}\}\right|.
\end{align*}
\begin{lemma}
\label{lem:modular-prelim} 
With notation as above
\begin{align*}
\begin{split}
    Q_{F, \lambda, m} \ll J_{F,\lambda,m}. 
\end{split}
\end{align*}
\end{lemma}
\begin{proof}
Suppose $\deg_F(u^2),\deg_F(v^2) < m$ satisfy $u-v \equiv \lambda \Mod{F}$. This implies 
$$u^2-v^2-\lambda^2 \equiv \lambda(u+v)-\lambda^2 \equiv \lambda(\lambda + 2v) - \lambda^2 \equiv 2\lambda v \Mod{F}. $$
Squaring
yields 
$$((u^2-v^2)-j\lambda^2)^2 \equiv 4\lambda^2 v^2 \Mod{F}. $$
Now we make the substitution 
$$u^2-v^2 \to x, ~v^2 \to z $$
to obtain 
\begin{align*}
    Q_{F, \lambda, m} \ll |\{&(x,z) \in \cI_m^2 : (x-\lambda^2)^2 \equiv 4\lambda^2 z \Mod{F}\}|.
\end{align*}
Now
$$(x-\lambda^2)^2 \equiv 4\lambda^2z \Mod{F} $$
implies, after the change of variables $2x + z \to y$, 
\begin{align*}
   x^2 + \lambda^4 \equiv 2\lambda^2y \Mod{F}  
\end{align*}
from which the desired result follows.
\end{proof}

We also define the lattice and convex body 
\begin{align}\label{eq:modular_roots_lattice_body}
\begin{split}
  \cL_\lambda &= \{(x,y) \in \F_{q}[T]^2 : x \equiv 2\lambda^2y \Mod{F}\},  \\
    \cB &= \{(x,y) \in \Ki^2 : |x| < q^{2m}, |y| < q^{m}\}.  
\end{split}
\end{align}
For a given $x \in \cI_r$, there is at most one $y \in \cI_m$ such that 
\begin{align}\label{eq:J_equation}
    x^2 + \lambda^4 \equiv 2\lambda^2y \Mod{F}.
\end{align}
Furthermore, for any two pairs $(x_1, y_1), (x_2, y_2)$ satisfying (\ref{eq:J_equation}) we have 
$$(x_1-x_2)(x_1+x_2) \equiv 2\lambda^2(y_1-y_2) \Mod{F}. $$
Thus applying Lemma \ref{lem:divisorbound} yields 
\begin{align}\label{eq:J_lattice_intersection}
    J_{F, \lambda, m}^2 \leq q^{o(m)}|\cL_\lambda \cap \cB| 
\end{align}
which will be used throughout this section.

We also need one more preliminary result for this section.
\begin{lem}\label{lem:Q_two_cases}
For any integer $m \leq r$ we have that for any $\lambda \in \cI_r\setminus\{0\}$ either 
\begin{align}
\label{eq:Q-two-cases}
Q_{F, \lambda, m} \leq q^{o(m)}\max\{q^{3m/2 - r/2}, 1\}, 
\end{align}
or there exists $a,b \in \F_q[T]$, coprime to $F$, with 
\begin{align*}
|a|  < \frac{q^{2m + o(m)}}{Q^2_{F, \lambda,m}}, ~|b| < \frac{q^{m+o(m)}}{Q^2_{F, \lambda,m}} 
\end{align*}
satisfying 
$$a\ov{b} \equiv 2\lambda^2 \Mod{F}. $$
\end{lem}
\begin{proof}
We let $\cL_\lambda$ and $\cB$ be as in (\ref{eq:modular_roots_lattice_body}). By Lemma~\ref{lem:modular-prelim} and (\ref{eq:J_lattice_intersection}) we have 
$$Q_{F, \lambda, m} \ll J_{F, \lambda, m} \leq q^{o(m)}|\cL_\lambda \cap \cB|^{1/2}.$$ 

We let $\sigma_{1, \lambda}$ and $\sigma_{2, \lambda}$ denote the successive minima of $\cL$ with respect to $\cB$. If $\sigma_{1, \lambda} >1$ then $|\cL_\lambda \cap \cB| = 1$ so $J_{F, \lambda, m} \leq q^{o(m)}$ which is accounted for in (\ref{eq:Q-two-cases}). Thus, we now assume $\sigma_{1, \lambda} \leq 1$. 

If $\sigma_{2, \lambda} \leq 1$ then Corollary \ref{cor:Mahler_general} and Lemma  \ref{lem:lattice_body_intersection} imply 
$$|\cL_{\lambda} \cap \cB | \ll q^{3m-r} $$
which implies (\ref{eq:Q-two-cases}). Finally if $\sigma_{2, \lambda}> 1$ then Lemma~\ref{lem:last} with $k=1,\beta=\infty$ and $\cS = \F_q[T]^d$ implies there exists non-zero $(a,b) \in \cL$ such that 
$$|a| < \frac{q^{2m + o(m)}}{Q^2_{F,\lambda,m}}, ~|b| < \frac{q^{m + o(m)}}{Q^2_{F,\lambda,m}}.$$
Since $m \leq r$ and we may assume that \eqref{eq:Q-two-cases} does not hold, this implies that $b$ is invertible modulo $F$, and this concludes the proof. 
\end{proof}

\subsection{Proof of Theorem~\ref{thm:energy_bound} } 

Before completing the proof of Theorem \ref{thm:energy_bound}, we need the following fourth moment estimate for $Q_{F, \lambda, m}$.

\begin{lem}\label{lem:Q_fourth_moment}
For any $m \leq r$ we have 
$$\sum_{\lambda \in \cI_r\setminus\{0\}}Q_{F,\lambda,m}^4 \leq q^{13m/2 - 3r/2 + o(m)} + q^{3m + o(m)}. $$
\end{lem}
\begin{proof}
Firstly we apply the dyadic pigeonhole principle. For each non-negative integer $\nu$ we define 
$$\Gamma_\nu = \{\lambda \in \cI_r\setminus\{0\}: 2^\nu \leq Q_{F, \lambda, m} \leq 2^{\nu +1}-1\}. $$
We may ignore the case $Q_{F, \lambda, m} = 0$ as this contributes nothing. Note there are at most $q^{o(m)}$ such $\nu$ for which $|\Gamma_\nu| \neq 0$, since we trivially have $Q_{F, \lambda, m} \ll q^{2m}$. For a given $\nu$ we have 
$$\sum_{\lambda \in \Gamma_\nu}Q_{F,\lambda,m}^4 \ll 2^{4\nu}|\Gamma_\nu|.$$
If we let $\nu_0$ be the integer for which this sum is maximized and set $\Delta = 2^{\nu_0}$, then we obtain 
\begin{align}\label{eq:Q_fourth_moment_bound1}
    \sum_{\lambda \in \cI_r\setminus\{0\}}Q_{F,\lambda,m}^4 \leq q^{o(r)}\Delta^{4}|\Gamma|
\end{align}
where 
$$\Gamma = \{\lambda \in \cI_r\setminus\{0\}: \Delta \leq Q_{F, \lambda, m} < 2\Delta\}. $$
We also trivially have 
$$\Delta |\Gamma| \ll \sum_{\lambda \in \Gamma}Q_{F, \lambda, m} \ll q^{2m}. $$

We now consider the two cases 
$$\Delta \leq q^{c(m)}\max\{q^{3m/2-r/2}, 1\}  $$
or 
$$\Delta >q^{c(m)}\max\{q^{3m/2-r/2}, 1\}  $$
for some sufficient function $c(m) = o(m)$. 

In the first case we obtain 
\begin{align}
\begin{split}\label{eq:Q_fourth_moment_case1}
        \Delta^4|\Gamma|
    &\leq q^{o(m)}(q^{9m/2-3r/2}+1)\Delta|\Gamma|  \\
    &\leq q^{o(m)}(q^{13m/2-3r/2}+q^{2m}).
\end{split}
\end{align}

In the second case, by Lemma \ref{lem:Q_two_cases} for each $\lambda \in \Gamma$ there exists $a,b \in \cI_r\setminus\{0\}$ such that 
$$|a| < \frac{q^{2m + o(m)}}{\Delta^2}, ~|b| < \frac{q^{m + o(m)}}{\Delta^2} $$
satisfying $2\lambda^2b \equiv a\Mod{F}$. If $a$ and $b$ are fixed then $\lambda$ can take at most two values. Thus 
$$|\Gamma| < \frac{q^{2m + o(m)}}{\Delta^2}\frac{q^{m + o(m)}}{\Delta^2} = \frac{q^{3m + o(m)}}{\Delta^4} $$
so we can say
\begin{align}\label{eq:Q_fourth_moment_case2}
    \Delta^4 |\Gamma| < q^{3m+o(m)}.
\end{align}

Together (\ref{eq:Q_fourth_moment_bound1}), (\ref{eq:Q_fourth_moment_case2}) and  (\ref{eq:Q_fourth_moment_case1}) yield 
\begin{align*}
    \sum_{\lambda \in \cI_r\setminus\{0\}}Q_{F,\lambda,m}^4 \leq q^{13m/2-3r/2 + o(m)} + q^{3m + o(m)}
\end{align*} 
as desired. 
\end{proof}

We can now complete the proof of Theorem \ref{thm:energy_1m_improved}. We will make use of a few expressions derived in \cite{BS2022}. By \cite[equation (3.3)]{BS2022} we have 

\begin{align}\label{eq:BS2022_3.3}
    \sum_{\lambda \in \cI_r}|Q_{F, \lambda}(\balpha)| \ll \|\balpha\|_1^2.
\end{align}

The Holder inequality also gives 
\begin{align}\label{eq:Q_sum_holder}
    \sum_{\lambda \in \cI_r\setminus\{0\}}|Q_{F, \lambda}(\balpha)|^2 \leq \bigg{(}\sum_{\lambda \in \cI_r}|Q_{F, \lambda}(\balpha)| \bigg{)}^{2/3} \bigg{(}\sum_{\lambda \in \cI_r}|Q_{F, \lambda}(\balpha)|^4 \bigg{)}^{1/3}. 
\end{align}
We also note the trivial inequality 
$$Q_{F, \lambda}(\balpha) \ll \|\balpha\|_\infty^2Q_{F, \lambda, m}. $$
Therefore Lemma \ref{lem:Q_fourth_moment} implies 
$$\sum_{\lambda \in \cI_r\setminus\{0\}}|Q_{F, \lambda}(\balpha)|^4 \leq \|\balpha\|_\infty^8q^{o(m)}(q^{13m/2-3r/2 } + q^{3m}). $$
Thus using this and (\ref{eq:BS2022_3.3}) in (\ref{eq:Q_sum_holder}) yields 
$$ \sum_{\lambda \in \cI_r\setminus\{0\}}|Q_{F, \lambda}(\balpha)|^2 \leq \|\balpha\|_1^{4/3}\|\balpha\|_\infty^{8/3}q^{o(m)}(q^{13m/6 - r/2} + q^m). $$

Substituting this into 
$$E_{F, 2}^{\sqrt}(\balpha) = \sum_{\lambda \in \cI_r\setminus\{0\}}|Q_{F, \lambda}(\balpha)|^2 + O\left(\|\balpha\|_2^4\right) $$
from \cite[page 10]{BS2022}
completes the proof, after noting that H{\"o}lder's inequality implies
 $$\|\balpha\|_2^4 \ll \|\balpha\|_\infty^{8/3}\|\balpha\|_1^{4/3}q^{2m/3}. $$

\subsection{Proof of Theorem~\ref{thm:energy_1m_improved}  } 
Firstly we note that 
$$E_{F, 2}^{\sqrt}(\bm{1}_m) = \sum_{\lambda \in \cI_r}Q_{F, \lambda, m}^2 $$
where we recall from (\ref{eq:Q_1m}) $$Q_{F, \lambda, m}^2 = \big{|}\{(u,v) \in \cI_r^2 : \deg_F(u^2), \deg_F(v^2) < m, u-v \equiv \lambda \Mod{F}\}\big{|}.$$
If $\lambda = 0$ then $Q_{F, \lambda, m}^2 \ll q^{2m}$, so now by Lemma~\ref{lem:modular-prelim}
$$E_{F, 2}^{\sqrt}(\bm{1}_m) \ll \sum_{\lambda \in \cI_r\setminus\{0\}}J_{F, \lambda, m}^2 + q^{2m}$$
where we recall
\begin{align*}
        J_{F, \lambda, m} = \big{|}\{(x,y) \in \cI_r^2: x^2 + \lambda^4 \equiv 2\lambda^2y \Mod{F}\}\big{|}. 
\end{align*}

We let $\cL_\lambda$ and $\cB$ be as in (\ref{eq:modular_roots_lattice_body}). Let $\sigma_{1,\lambda}$ and $\sigma_{2,\lambda}$ denote the successive minima of $\cL_\lambda$ with respect to $\cB$. We now partition summation as
$$E_{F, 2}^{\sqrt}(\bm{1}_m) \ll S_0 + S_1 + S_2  + q^{2m}$$
where 
$$S_0 = \sum_{\substack{\lambda \in \cI_r\setminus\{0\} \\ \sigma_{1, \lambda} > 1}}J_{F, \lambda, m}^2, ~~S_1 = \sum_{\substack{\lambda \in \cI_r\setminus\{0\} \\ \sigma_{1,\lambda} \leq 1  \\ \sigma_{2,\lambda} > 1}}J_{F, \lambda, m}^2,~~S_2 = \sum_{\substack{\lambda \in \cI_r\setminus\{0\} \\ \sigma_{2,\lambda} \leq 1}}J_{F, \lambda, m}^2.$$
We consider each sum separately.

For $S_0$, we have $\sigma_{1_\lambda} > 1$. Note that by Lemma \ref{lem:lattice_body_intersection} this implies $\cL\cap \cB = \{(0,0)\}$. Thus by (\ref{eq:J_lattice_intersection}), $J_{F, \lambda, m} \leq q^{o(m)} $ so this implies  
$$S_0 \leq q^{o(m)}\sum_{\substack{\lambda \in \cI_r }}J_{F, \lambda, m} \leq q^{2m + o(m)}. $$

To deal with $S_2$, if $\sigma_{2,\lambda} \leq 1$ then Corollary \ref{cor:Mahler_general}, Lemma \ref{lem:lattice_body_intersection} and Lemma \ref{lem:det_of_modular_lattice} imply 
$$|\cL_\lambda \cap \cB| \ll q^{3m-r}. $$
Thus again applying (\ref{eq:J_lattice_intersection}),
$$J_{F, \lambda, m} \leq q^{3m/2-r/2+o(m)} $$
which implies 
$$S_2 \leq  q^{3m/2-r/2+o(m)}\sum_{\substack{\lambda \in \cI_r }}J_{F, \lambda, m} \leq q^{7m/2 - r/2 + o(m)}.$$

We finally consider $S_1$, where we assume $\lambda$ satisfies $\sigma_{1,\lambda} \leq 1$ and $\sigma_{2,\lambda} > 1$. Of course for any such $\lambda$, we can assume $J_{F, \lambda, m} > 0$ otherwise it contributes nothing to $S_1$. So let $(x_\lambda, y_\lambda)$ satisfy (\ref{eq:J_equation}) with $|x_\lambda|,|y_\lambda| < q^m$. Then we would have that 
$$J_{F, \lambda, m} \ll |\{(x,y) \in \F_q[T]^2:  (x^2-x_\lambda^2, y-y_\lambda) \in \cL_\lambda \cap 
\cB\}|. $$
Let $x^{(1)}_\lambda$ be the vector corresponding to $\sigma_{1,\lambda}$. Since $\sigma_{1,\lambda} \leq 1$ we know $x_\lambda^{(1)} \in \cB$, so if $x^{(1)}_\lambda = (a_\lambda, b_\lambda)$ then $a_\lambda, b_\lambda \neq 0$ and $\gcd(a_{\lambda}, b_{\lambda}) = 1$. 
But further, since $\sigma_{1,\lambda} \leq 1$ and $\sigma_{2,\lambda} > 1$, this implies that all points in the intersection $\cL_\lambda \cap \cB$ must be an $\F_q[T]$-multiple of $x_\lambda^{(1)}$ and thus 
$$J_{F, \lambda, m} \ll K_{\lambda, m} + 1 $$
where 
$$K_{\lambda, m} = |\{(x,y) \in \cI_m^2: y \neq y_{\lambda},~ \frac{x^2-x_\lambda^2}{y-y_\lambda} = \frac{a_\lambda}{b_\lambda}\}|.$$
With this in mind we write
\begin{align}\label{eq:S1_intermediate_JandK}
    S_1
    &\ll \sum_{\lambda \in \cI_r\setminus\{0\}}J_{F, \lambda, m}(K_{\lambda, m} + 1) \leq\sum_{\lambda \in \cI_r\setminus\{0\}}J_{F, \lambda, m}K_{\lambda, m} + q^{2m}  .
\end{align}
For a given $\lambda$ suppose we have
$$\frac{x^2-x_{\lambda}^2}{y-y_{\lambda}}=\frac{a_\lambda}{b_\lambda}. $$
We consider when $x$ is fixed, and then when $y$ is fixed. 

If $y$ is fixed, then $x$ is defined in at most two ways. Since $\gcd(a_\lambda,b_\lambda)=1$ we have 
$$y - y_\lambda \equiv 0 \Mod{b_\lambda} $$
so there are at most $q^{m-\deg b_\lambda}$ possibilities for $y$. Thus  
\begin{align}\label{eq:b_lambda_bound}
    K^2_{\lambda, m} \ll q^{2m - 2\deg b_\lambda}.
\end{align}
If $x$ is fixed, $y$ is uniquely defined. Since $\gcd(a_\lambda,b_\lambda)=1$ we have 
$$x^2 - x_{\lambda}^2 \equiv 0 \mod{a_\lambda} $$
which we write as 
$$(x - x_{\lambda})(x+x_{\lambda}) \equiv 0 \mod{a_\lambda} $$
So there are $a_1,a_2 \in \cI_m$ satisfying 
$$a_1a_2 = a_\lambda $$
such that 
$$x \equiv x_{\lambda} \Mod{a_1},~x \equiv -x_{\lambda}\Mod{a_2}. $$
Therefore if $N$ denotes the number of possibilities for $x$ then 
$$N \ll q^{m- \deg\text{lcm}(a_1,a_2)} + 1.$$
Thus 
\begin{align*}
    K_{\lambda, m} 
    &\ll  \sum_{a_1,a_2 = a}\big{(}q^{m + \deg\gcd(a_1,a_2) - \deg a_\lambda} + 1 \big{)} \\
    &\ll q^{o(m)} + \sum_{a_1,a_2 = a_\lambda}q^{m + \deg\gcd(a_1,a_2) - \deg a_\lambda} 
\end{align*}
where we have again used Lemma \ref{lem:divisorbound}. Therefore if we let $L$ denote the set of all $\lambda \in \cI_r\setminus\{0\}$ such that $\sigma_{1,\lambda} \leq 1$ and $\sigma_{2,\lambda} > 1$, then there exists an absolute constant $c$ and a fixed function $C(m) = o(m)$ such that $L = L_1 \cup L_2$ where 
$$L_1 = \{\lambda \in \cI_r\setminus\{0\} : K_{\lambda, m} \leq c \sum_{a_1,a_2 = a_\lambda}q^{m + \deg\gcd(a_1,a_2) - \deg a_\lambda}\}$$
and
$$L_2 = \{\lambda \in \cI_r\setminus\{0\} : K_{\lambda, m} \leq q^{C(m)}\}. $$
From (\ref{eq:S1_intermediate_JandK}), and after recalling $J_{F, \lambda, m} \ll K_{\lambda, m} + 1$, we can then write $S_1 \ll R_1 + R_2 + q^{2m}$ where 
\begin{align*}
    R_1 = \sum_{\lambda \in L_1}K_{\lambda, m}^2, ~~R_2 = \sum_{\lambda \in L_2}K_{\lambda, m}J_{F, \lambda, m}. 
\end{align*}

To firstly deal with $R_2$, we have simply
$$R_2 \leq q^{o(m)}\sum_{\lambda \in \cI_r}J_{F, \lambda, m} \leq q^{2m + o(m)}. $$

Next to deal with $R_1$, note for $\lambda \in L_1$ we have by Cauchy-Schwarz and again Lemma \ref{lem:divisorbound}
\begin{align*}
    K_{\lambda, m}^2 
    &\ll q^{2m-2\deg a}\bigg{(} \sum_{a_1,a_2 = a}q^{\deg\gcd(a_1,a_2)}\bigg{)}^2\\
    &\leq q^{2m-2\deg a}\sum_{a_1'a_2'=a}1\sum_{a_1,a_2 = a}q^{2\deg\gcd(a_1,a_2)}\\
    &\leq q^{2m+ o(m)}\sum_{a_1a_2 = a}q^{2\deg\gcd(a_1,a_2)-2\deg a}.
\end{align*}
From the definitions of $\cL_{\lambda}$ and $\cB$, if $\lambda \neq \lambda'$ then $(a_\lambda, b_\lambda) \neq (a_{\lambda'}, b_{\lambda'})$. Thus recalling (\ref{eq:b_lambda_bound}) we have 

\begin{align*}
    R_1
    &\leq q^{2m + o(m)}\sum_{\substack{|a|< q^{2m} \\ |b| < q^m}}\sum_{\substack{a_1a_2 = a \\ |a_i| < q^m}}\min\bigg{\{}q^{-2\deg b},q^{2\deg\gcd(a_1,a_2)-2\deg a} \bigg{\}}\\
    &\leq q^{2m + o(m)}\sum_{|a_1|,|a_2|,|b| < q^m}\min\bigg{\{}q^{-2\deg b},q^{2\deg\gcd(a_1,a_2)-2\deg a_1-2\deg a_2} \bigg{\}}\\
    &\leq q^{2m + o(m)}\mathop{\sum\sum}\limits_{\substack{|t|< q^m, ~|b| < q^m \\ |a_1|,|a_2| < q^m \\ \gcd(a_1,a_2)=t}}\min\bigg{\{}q^{-2\deg b},q^{2\deg t-2\deg a_1-2\deg a_2} \bigg{\}}\\
    &\leq q^{2m + o(m)}\sum_{|t|< q^m}\sum_{|b|< q^m}\sum_{\substack{\substack{|a_1|,|a_2| < q^{m - \deg t}}}}\min\bigg{\{}q^{-2\deg b},q^{-2\deg(a_1a_2t)} \bigg{\}}.
\end{align*}
Once again using Lemma \ref{lem:divisorbound} we obtain 
\begin{align*}
    R_1 
    &\leq q^{2m + o(m)}\sum_{|b| < q^m}\sum_{|a| < q^{2m}}\min\bigg{\{}q^{-2\deg b},q^{-2\deg a} \bigg{\}}\\
    &\leq q^{2m + o(m)}\bigg{(}\sum_{|b|< q^m}\sum_{\deg a \leq \deg b}q^{-2\deg b} + \sum_{|a| < q^{2m}}\sum_{\deg b \leq \deg a}q^{-2\deg a} \bigg{)}\\
    &\leq q^{2m + o(m)}.
\end{align*}
Our estimates for $R_1$ and $R_2$ thus yield $S_1 \leq q^{2m + o(m)}$.

Combining our estimates for $S_0, S_1$ and $S_2$ gives the desired bound for $E_{F, 2}^{\sqrt}(\bm{1}_m). $

\subsection{Proof of Theorem~\ref{thm:T4_bound} } In preparation we firstly prove the following.

\begin{lem}\label{lem:K_delta}
Let $\epsilon > 0$ and $\Delta \geq 1$. For any  integer $m \leq r$, let 
$$\cD = \{\lambda \in \cI_r\setminus\{0\} ~:~ Q_{F, \lambda, m} \geq \Delta\}.$$
Let $K > 1$ and suppose we have 
\begin{align}\label{eq:K_lowerbound}
  K \geq q^{\epsilon m}\big{(}\frac{q^{15m/2-r/2}}{\Delta^{12}}+\frac{q^{5m-r/4}}{\Delta^{8}}\big{)}   
\end{align}
and 
\begin{align}\label{eq:delta_lowerbound}
    \Delta \geq q^{\epsilon m}(q^{3m/2-r/2} + q^{5m/8-r/8}). 
\end{align}
Further we define $\cF$ to be the set of $d \in \cI_r\setminus\{0\}$ such that 
\begin{align*}
     K \leq |\{(\lambda_1, \lambda_2) \in \cD^2 ~:~d = \lambda_1-\lambda_2 \}|.
\end{align*}

Then either 
$$K \ll 1 $$
or 
$$K|\cF| \ll \frac{q^{3m + o(m)}}{\Delta^4}. $$
\end{lem}

\begin{proof}
Firstly, for a fixed $d \in \cF$ if $\lambda_1, \lambda_2 \in \cD$ satisfy $\lambda_1 - \lambda_2 = d$ then 
\begin{align*}
    \lambda_1^2 - \lambda_2^2 - d^2  &\equiv (\lambda_1-\lambda_2)^2 + 2\lambda_1\lambda_2 - 2\lambda_2^2 - d^2 \\ 
    &\equiv 2\lambda_2(\lambda_1-\lambda_2) \equiv 2\lambda_2d ~~~~~\Mod{F}.
\end{align*}
Squaring and multiplying by $4$ then yields 
$$(2\lambda_1^2 - 2\lambda_2^2 - 2d^2)^2 \equiv 8d^2(2\lambda_2^2) \Mod{F}. $$
Now (\ref{eq:delta_lowerbound}) implies $Q_{F, \lambda}(\bm{1}_m) \geq q^{\epsilon m + 3m/2-r/2}$, so by Lemma \ref{lem:Q_two_cases}, for significantly large $m$ there exists $a_{\lambda}, b_{\lambda} \in \F_q[T]$ such that 
\begin{align}\label{eq:alambda_blambda_bounds}
    |a_\lambda| \leq \frac{q^{2m +o(m)}}{\Delta^2}, ~|b_\lambda| \leq \frac{q^{m + o(m)}}{\Delta^2}
\end{align}
and $a_{\lambda}b_{\lambda}^{-1} \equiv 2\lambda^2 \Mod{F}$, such that $\gcd(a_\lambda, b_\lambda) = 1$. Thus if we let 
\begin{align}\label{eq:I(d)}
    I(d) = \{(\lambda_1, \lambda_2) \in \cD :
    (a_{\lambda_1}&b_{\lambda_1}^{-1} - a_{\lambda_2}b_{\lambda_2}^{-1} - 2d^2)^2 \\
    &\equiv 8d^2a_{\lambda_2}b_{\lambda_2}^{-1} \Mod{F}\}\nonumber
\end{align}
then 
$$K \leq I(d). $$
Multiplying the congruence in (\ref{eq:I(d)}) by $b_{\lambda_1}^2b_{\lambda_2}^2$ gives 
$$(a_{\lambda_1}b_{\lambda_2} - a_{\lambda_2}b_{\lambda_1}-2d^2b_{\lambda_1}b_{\lambda_2})^2 \equiv 8d^2a_{\lambda_2}b_{\lambda_1}^2b_{\lambda_2} \Mod{F}$$
and rearranging then yields 
\begin{align}\label{eq:equation_for_triple_in_intersection}
    (a_{\lambda_1}b_{\lambda_2}-a_{\lambda_2}b_{\lambda_1})^2 -4d^2b_{\lambda_1}b_{\lambda_2}(&a_{\lambda_1}b_{\lambda_2}+a_{\lambda_2}b_{\lambda_1})\\\nonumber 
    &+4d^4(b_{\lambda_1}b_{\lambda_2})^2 \equiv 0 \Mod{F}.
\end{align}
Now let
$$\cL = \{(x,y,z) \in \F_q[T]^3~:~ x + yd^2 + zd^4 \equiv 0 \Mod{F}\} $$
and 
$$\cB = \{(x,y,z) ~:~|x| \leq \frac{q^{6m + C_0(m)}}{\Delta^8}, |y| \leq \frac{q^{5m + C_0(m)}}{\Delta^8}, |z| \leq \frac{q^{4m + C_0(m)}}{\Delta^8}\} $$
for a suitable function $C_0(m) = o(m)$. Using (\ref{eq:alambda_blambda_bounds}) and  (\ref{eq:equation_for_triple_in_intersection}) we obtain
\begin{align*}
    ((a_{\lambda_1}b_{\lambda_2}-a_{\lambda_2}b_{\lambda_1})^2, -4b_{\lambda_1}b_{\lambda_2}(a_{\lambda_1}b_{\lambda_2}+a_{\lambda_2}b_{\lambda_1}),
    4(b_{\lambda_1}b_{\lambda_2})^2) \in \cL \cap \cB
\end{align*} 
so $|\cL \cap \cB| \geq I(d)$. 
Let $\sigma_1, \sigma_2, \sigma_3$ denote the successive minima of $\cL$ with respect to $\cB$. Since $K \leq I(d) \leq |\cL \cap \cB|$, and since we have assumed $K > 1$, this implies $\sigma_1 \leq 1$. 

If we assume that $\sigma_2 > 1$, then every point in $\cL \cap \cB$ is an $\F_q[T]$ multiple of the vector corresponding to $\sigma_1$. So there exists some fixed $x_0, y_0, z_0 \in \F_q[T]$ such that for any $\lambda_1, \lambda_2 \in \cD$, 
\begin{align*}
    ((a_{\lambda_1}b_{\lambda_2}-a_{\lambda_2}b_{\lambda_1})^2, -4b_{\lambda_1}b_{\lambda_2}(a_{\lambda_1}b_{\lambda_2}+a_{\lambda_2}b_{\lambda_1}),
    4(b_{\lambda_1}b_{\lambda_2})^2)  = s(x_0, y_0, z_0)
\end{align*}
for some $s \in \F_q[T]$. Recall Lemma \ref{lem:Q_two_cases} ensures that $b_{\lambda_1}b_{\lambda_2} \neq 0$ since they are both coprime to $F$, so this means $z_0 \neq 0$. Thus 
$$\bigg{(}\frac{a_{\lambda_1}}{b_{\lambda_1}} - \frac{a_{\lambda_2}}{b_{\lambda_2}} \bigg{)}^2 = \frac{x_0}{z_0}$$
and
$$\frac{a_{\lambda_1}}{b_{\lambda_1}} + \frac{a_{\lambda_2}}{b_{\lambda_2}} = \frac{y_0}{z_0}.  $$
Thus 
\begin{align*}
    K \ll \bigg{|}\bigg{\{} (\lambda_1, \lambda_2) \in \cD : \bigg{(}\frac{a_{\lambda_1}}{b_{\lambda_1}} - &\frac{a_{\lambda_2}}{b_{\lambda_2}} \bigg{)}^2 = \frac{x_0}{z_0}, \\
    &\frac{a_{\lambda_1}}{b_{\lambda_1}} + \frac{a_{\lambda_2}}{b_{\lambda_2}} = \frac{y_0}{z_0}\bigg{\}}\bigg{|}.
\end{align*}
Note that if a square-root for $x_0/z_0$ exists, then this system of equations implies that $a_{\lambda_1}/b_{\lambda_1}$ and $a_{\lambda_2}/b_{\lambda_2}$ can take at most 2 values each. Since these are reduced fractions, there are thus at most $2q$ possibilities for each pair $(a_{\lambda_i}, b_{\lambda_i}).$ In particular, this implies 
$$K \ll 1. $$

We now suppose that $\sigma_2 \leq 1$. Let $G(x,y,z)$ count the number of solutions to 
$$b_1b_2 = z, ~ a_1b_2+a_2b_1 = y,~a_1b_2 - a_2b_1 = x$$
with 
$$|b_1|,|b_2| \leq \frac{q^{m + o(m)}}{\Delta^2}, ~|a_1|,|a_2| \leq \frac{q^{2m + o(m)}}{\Delta^2}.  $$
Note that given $z$, there are at most $q^{o(m)}$ possibilities for $(b_1, b_2)$ by Lemma \ref{lem:divisorbound}. Thus $G(x,y,z) \leq q^{o(m)}$. By (\ref{eq:alambda_blambda_bounds}) and (\ref{eq:equation_for_triple_in_intersection}) this implies 

\begin{align}\label{eq:K_intermsof_J}
    K  &\ll \sum_{\substack{|y|,|x| \leq q^{3m + C_0(m)}/\Delta^4 \\ |z| \leq q^{2m + C_0(m)}/\Delta^4 \\ x^2- 4d^2zy + 4d^4z^2 \equiv 0 \Mod{F}}} G(x,y,z) \nonumber \\
    &\leq q^{o(m)}|\{(x,y,z) \in \F_q[T]^3 : |y| \leq \frac{q^{5m + C_0(m)}}{\Delta^8},~ |x| \leq \frac{q^{3m + C_0(m)}}{\Delta^4},\nonumber\\ 
    &\quad\quad\quad\quad\quad\quad|z| \leq \frac{q^{2m + C_0(m)}}{\Delta^4},~
     x^2+ d^2y + d^4z^2 \equiv 0 \Mod{F}\}|.
\end{align}
For any fixed $x,z \in \F_q[T]$, (\ref{eq:delta_lowerbound}) implies that for large enough $m$ there is at most one value of $y$ that satisfies (\ref{eq:K_intermsof_J}). Also for any two solutions $(x_1, y_1, z_1), (x_2, y_2, z_2)$ to (\ref{eq:K_intermsof_J}) we have 
$$x_1^2- x_2^2 + d^2(y_1-y_2) + d^4(z_1^2 - z_2^2) \equiv 0 \Mod{F}. $$
Again using Lemma \ref{lem:divisorbound}, and keeping in mind that $y_1,y_2$ are uniquely defined by the other variables, we can say 
$$K^2 \leq q^{o(m)}|\cL \cap \cB| $$
where $\cL$ and $\cB$ are as before. 
Noting (\ref{eq:K_lowerbound}) and that in this case $\sigma_2 \leq 1$,  Lemma \ref{lem:intersection_3_cases} implies there exists some $x,y,z \neq 0$ such that 
\begin{align*}
    |x| \leq \frac{q^{9m + C_2(m)}}{\Delta^{16}K^2}, ~|y| \leq \frac{q^{10m + C_2(m)}}{\Delta^{16}K^2},~|z| \leq \frac{q^{11m + C_2(m)}}{\Delta^{16}K^2}
\end{align*}
for a suitable function $C_2(m) = o(m)$, where $x,y,z$ additionally satisfy 
\begin{align}\label{eq:relating_xyz}
   xd^2 \equiv y \Mod{F}, ~xd^4 \equiv z \Mod{F}. 
\end{align}
Since all terms are coprime to $F$, we may assume that $\gcd(x,y,z) = 1$.

From this point forward, these values $x,y,z$ are fixed and depend on $d$. Recall that $K$ is bounded by the number of $\lambda_1, \lambda_2$ satisfying (\ref{eq:equation_for_triple_in_intersection}). Substituting (\ref{eq:relating_xyz}) into (\ref{eq:equation_for_triple_in_intersection}) yields 
\begin{align*}
    x(a_{\lambda_1}b_{\lambda_2}-a_{\lambda_2}b_{\lambda_1})^2 -4yb_{\lambda_1}b_{\lambda_2}(&a_{\lambda_1}b_{\lambda_2}+a_{\lambda_2}b_{\lambda_1})\\
    &+4z(b_{\lambda_1}b_{\lambda_2})^2 \equiv 0 \Mod{F}.
\end{align*}
(\ref{eq:K_lowerbound}) implies that for large enough $m$, all of the terms in this equation have degree less than $\deg F$. Thus we can in fact say 
\begin{align*}
    x(a_{\lambda_1}b_{\lambda_2}-a_{\lambda_2}b_{\lambda_1})^2 -4yb_{\lambda_1}b_{\lambda_2}(&a_{\lambda_1}b_{\lambda_2}+a_{\lambda_2}b_{\lambda_1})\\
    &+4z(b_{\lambda_1}b_{\lambda_2})^2 =0.
\end{align*}
Similarly, (\ref{eq:relating_xyz}) gives $y^2 \equiv xz \Mod{F}$. Again, (\ref{eq:K_lowerbound}) ensures all of these terms have degree less than $\deg F$, and thus $y^2 = xz$. Combining these yields 
$$\bigg{(}\frac{a_{\lambda_1}}{b_{\lambda_1}} - \frac{a_{\lambda_2}}{b_{\lambda_2}}\bigg{)}^2 - 4\frac{y}{x} \bigg{(} \frac{a_{\lambda_1}}{b_{\lambda_1}} + \frac{a_{\lambda_2}}{b_{\lambda_2}} \bigg{)}  + 4\bigg{(}\frac{y}{x} \bigg{)}^2 =0.$$
Thus
\begin{align}\label{eq:withsqrt}
    \frac{y}{x} = \frac{1}{2}\bigg{(} \frac{a_{\lambda_1}}{b_{\lambda_1}} + \frac{a_{\lambda_2}}{b_{\lambda_2}} \bigg{)}  + \frac{(a_{\lambda_1}b_{\lambda_1}a_{\lambda_2}b_{\lambda_2})^{1/2}}{b_{\lambda_2}b_{\lambda_1}}
\end{align}
with the obvious caveat that $(a_{\lambda_1}b_{\lambda_1}a_{\lambda_2}b_{\lambda_2})^{1/2}$ either does not exist or takes on two possible values. Therefore 
\begin{align*}
    K \ll \bigg{|} \bigg{\{} (a_{\lambda_1}, &a_{\lambda_2}, b_{\lambda_1}, b_{\lambda_2}) \in \F_{q}[T]^4: \\
    &|b_1|,|b_2| \leq \frac{q^{m + C_3(m)}}{\Delta^2}, ~|a_1|,|a_2| \leq \frac{q^{2m + C_3(m)}}{\Delta^2} \\
    &\quad 0 \neq (a_{\lambda_1}b_{\lambda_1}a_{\lambda_2}b_{\lambda_2})^{1/2} \text{  exists and (\ref{eq:withsqrt}) holds}\bigg{\}}\bigg{|}
\end{align*}
for a suitable function $C_3(m) = o(m)$. Now if we sum the above over all $d \in \cF$, we note that for a given $x,y,z$ there are at most 2 possible values for $d$ by (\ref{eq:relating_xyz}).  Thus 
\begin{align*}
    K|\cF| &\ll \sum_{\substack{x,y \\ \gcd(x,y) = 1}}\bigg{|}\bigg{\{} (a_{\lambda_1}, a_{\lambda_2}, b_{\lambda_1}, b_{\lambda_2}) \in \F_{q}[T]^4: \\
    &\quad\quad\quad\quad \quad \quad  |b_1|,|b_2| \leq \frac{q^{m + C_3(m)}}{\Delta^2}, ~|a_1|,|a_2| \leq \frac{q^{2m + C_3(m)}}{\Delta^2} \\
    &\quad\quad \quad\quad\quad \quad 0 \neq (a_{\lambda_1}b_{\lambda_1}a_{\lambda_2}b_{\lambda_2})^{1/2} \text{  exists and (\ref{eq:withsqrt}) holds}\bigg{\}}\bigg{|}\\
    &\ll \bigg{|}\bigg{\{} (a_{\lambda_1}, a_{\lambda_2}, b_{\lambda_1}, b_{\lambda_2}) \in \F_{q}[T]^4: \\
    &\quad\quad \quad \quad\quad\quad   |b_1|,|b_2| \leq \frac{q^{m + C_3(m)}}{\Delta^2}, ~|a_1|,|a_2| \leq \frac{q^{2m + C_3(m)}}{\Delta^2} \\
    &\quad\quad \quad \quad\quad\quad \quad \quad \quad\quad  a_{\lambda_1}b_{\lambda_1}a_{\lambda_2}b_{\lambda_2} = v^2 \text{  for some } v \in \F_q[T]\bigg{\}}\bigg{|}\\ 
\end{align*}
where the second line comes from the fact that the quadruple $(a_{\lambda_1}, a_{\lambda_2},$ $b_{\lambda_1}, b_{\lambda_2})$ uniquely determines $(x,y)$ (up to a constant) since we can assume this fraction  $y/x$ is reduced. Again using Lemma \ref{lem:divisorbound} we can thus conclude 
\begin{align*}
    K|\cF| 
    &\leq q^{o(m)}\bigg{|}\bigg{\{} |x| \leq \frac{q^{6m + o(m)}}{\Delta^8} : ~x = v^2 \text{ for some } v \in \F_q[T] \bigg{\}}\bigg{|}\\
    &\leq \frac{q^{3m + o(m)}}{\Delta^4}
\end{align*}
as desired. 
\end{proof}

We can now complete the proof of Theorem \ref{thm:T4_bound}. We will assume that $m < r/3$, since otherwise (\ref{eq:T4_trivial}) is stronger than the desired bound. 
We recall the notation 
\begin{align*}
\begin{split}
        Q_{F, \lambda, m} &= |\{(u,v) \in \cI_r^2:~\deg_F(u^2),\deg_F(v^2) < m,~u-v \equiv \lambda \Mod{F}\}| \\
    &= |\{(u,v) \in \cI_r^2:~\deg_F(u^2),\deg_F(v^2) < m,~u+v \equiv \lambda \Mod{F}\}|. 
\end{split}
\end{align*}
Using this we can write
\begin{align*}
    E_{F, 4}^\sqrt(\bm{1}_m)
    &= \sum_{\lambda \in \cI_r}\bigg{(}\sum_{d \in \cI_r}Q_{F,d, m}Q_{F, \lambda-d, m} \bigg{)}^2.
\end{align*}
We now apply the dyadic pigeonhole principle. For any non-negative integer $\nu$ we define 
$$\cD_\nu = \{d \in \cI_r~:~ 2^\nu \leq Q_{F,d, m} < 2^{\nu+1}\}. $$
Again, we can ignore the case $Q_{F,d,m} = 0$ since this contributes nothing. For any $d \in \cI_r$ we define the characteristic function
$$\cD_\nu(d) = 
\begin{cases}
1, &d \in \cD_\nu\\
0, &d \not\in \cD_\nu.
\end{cases}
$$
Then for any pair $(u,v)$ we have 
\begin{align}\label{eq:DPHP_d_lambda}
\begin{split}
        \sum_{\lambda \in \cI_r}\bigg{(}\sum_{\substack{d \in \cI_r \\ d \in \cD_u \\ \lambda-d \in \cD_v}}&Q_{F,d, m}Q_{\lambda-d, m} \bigg{)}^2  \\ &\ll 2^{2u + 2v}\sum_{\lambda \in \cI_r}\bigg{(}\sum_{\substack{d \in \cI_r}}\cD_u(d)\cD_v(\lambda-d)\bigg{)}^2.
\end{split}
\end{align}
Of course for any given $\lambda$, every $d \in \cI_r$ satisfies $d \in \cD_u$ and $\lambda -d \in \cD_v$ for some pair $(u,v)$. This implies
\begin{align*}
    \sum_{\lambda \in \cI_r}\bigg{(}\sum_{\substack{d \in \cI_r}}&Q_{F,d, m}Q_{\lambda-d, m} \bigg{)}^2  &\ll \sum_{(u,v)} \sum_{\lambda \in \cI_r}\bigg{(}\sum_{\substack{d \in \cI_r \\ d \in \cD_u \\ \lambda-d \in \cD_v}}&Q_{F,d, m}Q_{\lambda-d, m} \bigg{)}^2,
\end{align*}
The trivial bound $Q_{F,d, m} \ll q^{2m}$ implies that there are at most $q^{o(m)}$ pairs $(u,v)$ such that $|\cD_v| \neq 0$ and $|\cD_u| \neq 0$. Thus if we let $(u,v)$ be the pair for which (\ref{eq:DPHP_d_lambda}) is maximized we have 
$$ E_{F, 4}^\sqrt(\bm{1}_m) \leq q^{o(m)}2^{2v + 2u}\sum_{\lambda \in \cI_r}\bigg{(}\sum_{\substack{d \in \cI_r}}\cD_u(d)\cD_v(\lambda-d)\bigg{)}^2. $$
Now we see simply that 
\begin{align*}
    \bigg{(}\sum_{\lambda \in \cI_r}\bigg{(}\sum_{\substack{d \in \cI_r}}\cD_u(d)\cD_v(\lambda-d)&\bigg{)}^2\bigg{)}^2\\
    \leq
    \bigg{(}\sum_{\lambda \in \cI_r}&\bigg{(}\sum_{\substack{d \in \cI_r}}\cD_u(d)\cD_u(\lambda-d)\bigg{)}^2\bigg{)}\\
    &\bigg{(}\sum_{\lambda \in \cI_r}\bigg{(}\sum_{\substack{d \in \cI_r}}\cD_v(d)\cD_v(\lambda-d)\bigg{)}^2\bigg{)}
\end{align*}
which implies
\begin{align}\label{eq:DPHP_D}
    E_{F, 4}^\sqrt(\bm{1}_m) \leq q^{o(m)}\Delta^{4}\sum_{\lambda \in \cI_r}\bigg{(}\sum_{\substack{d \in \cI_r}}\cD(d)\cD(\lambda-d)\bigg{)}^2
\end{align}
for some $\cD, \Delta$ satisfying 
$$\cD = \{d \in \cI_r ~:~\Delta \leq Q_{F,d,m} < 2\Delta\}. $$
Of course note that $\Delta \neq 0$. We recall the notation
$$(\cD * \cD)(\lambda) = \sum_{\substack{d \in \cI_r}}\cD(d)\cD(\lambda-d) $$
and we write 
$$S(\cD) = \sum_{\lambda \in \cI_r}(\cD * \cD)(\lambda)^2 $$
so that (\ref{eq:DPHP_D}) becomes 
\begin{align}\label{eq:T_intermsof_S}
    E_{F, 4}^\sqrt(\bm{1}_m) \leq q^{o(m)}\Delta^{4}S(\cD).
\end{align}
Note that $(\cD * \cD)(0) = |\cD|$
since $Q_{F,d, m} = Q_{F, -d, m}$, so $d \in \cD$ if and only if $-d \in \cD$. Thus 
$$S(\cD) = |\cD|^2 + \sum_{\lambda \in \cI_r\setminus\{0\}}(\cD * \cD)(\lambda)^2. $$
If we assume $S(\cD) \leq 2|\cD|^2$, combining this and (\ref{eq:T_intermsof_S}) with
$$\Delta^2|\cD| \leq \sum_{d \in \cD}Q_{F,d,m}^2 \leq E_{F, 2}^{\sqrt}(\bm{1}_m)$$
yields 
$$E_{F, 4}^\sqrt(\bm{1}_m) \leq q^{o(m)}\Delta^4|\cD|^2 \leq q^{o(m)}E_{F, 2}^{\sqrt}(\bm{1}_m)^2. $$
Now applying Theorem \ref{thm:energy_1m_improved} implies the desired bound. 

Thus, we now assume 
$$S(\cD) \leq 2\sum_{\lambda \in \cI_r\setminus\{0\}}(\cD * \cD)(\lambda)^2. $$
We again apply the dyadic pigeon hole principle. For any non-negative integer $w$ we define 
$$\cF_w = \{\lambda \in \cI_r\setminus\{0\} ~:~ 2^w \leq (\cD * \cD)(\lambda) < 2^{w+1}\}, $$
ignoring the case $(\cD * \cD)(\lambda) = 0$ as this contributes nothing. Then for a given $w$ we have
$$S(\cD) \leq 2\sum_{\lambda \in \cF_w}(\cD * \cD)(\lambda)^2 \ll 2^{2w}|\cF_w|.  $$
Now there are at most $q^{o(m)}$ such $w$ for which this sum is non-zero since $(\cD * \cD)(f) \leq |\cD| \leq \Delta|\cD| \leq q^{2m}$. Thus if the above sum is maximized at some $w \in \Z$, if we let $K = 2^{w+1}$ we have that 
$$S(\cD) \leq q^{o(m)}K^2|\cF|$$
where
$$\cF = \{\lambda \in \cI_r\setminus\{0\}~:~K \leq (\cD * \cD)(\lambda) < 2K\}. $$
Thus (\ref{eq:T_intermsof_S}) yields 
$$E_{F, 4}^\sqrt(\bm{1}_m) \leq q^{o(m)}\Delta^4K^2|\cF|. $$
We now fix some $\epsilon > 0$. Firstly, we assume that either 
$$K < q^{\epsilon m}\big{(}\frac{q^{15m/2-r/2}}{\Delta^{12}}+\frac{q^{5m-r/4}}{\Delta^{8}}\big{)}  $$
or
$$\Delta < q^{\epsilon m}(q^{3m/2-r/2} + q^{5m/8-r/8}). $$
We trivially have 
$$K |\cF| \leq |\cD|^2~ \text{ and } ~ \Delta|\cD| \ll q^{2m}$$
so in the first case using the assumed bound on $K$ we have
\begin{align*}
    E_{F, 4}^\sqrt(\bm{1}_m) 
    &\leq q^{o(m)}\Delta^4K|\cD|^2\\
    &\leq q^{4m + o(m)}\Delta^2K\\
    &\leq q^{4m + \epsilon m + o(m)}\big{(}\frac{q^{15m/2-r/2}}{\Delta^{10}}+\frac{q^{5m-r/4}}{\Delta^{6}}\big{)}\\
    &\leq q^{6m + \epsilon m + o(m)}\big{(}{q^{11m/2-r/2}}+{q^{3m-r/4}}\big{)}.
\end{align*}
Since we have assumed that $m < r/3$, then Theorem \ref{thm:energy_1m_improved} implies 
$$\Delta^2|\cD| \leq E_{F, 2}^{\sqrt}(\bm{1}_m) \leq q^{2m + o(m)}. $$
So in the second case, using $K \leq |\cD|$ and our assumed bound for $\Delta$ we have
\begin{align*}
    E_{F, 4}^\sqrt(\bm{1}_m) 
    &\leq q^{o(m)}\Delta^4|\cD|^3\\
    &\leq q^{6m + o(m)}\Delta\\
    &\leq (q^{3m/2-r/2} + q^{5m/8-r/8})q^{6m + \epsilon m + o(m)}.
\end{align*}
Combining these two estimates for $ E_{F, 4}^\sqrt(\bm{1}_m) $ yields 
$$E_{F, 4}^\sqrt(\bm{1}_m) \leq q^{6m + \epsilon m + o(m)}\big{(}{q^{11m/2-r/2}}+{q^{3m-r/4}} + q^{5m/8-r/8}\big{)}. $$

On the other hand, suppose that 
$$K \geq q^{\epsilon m}\big{(}\frac{q^{15m/2-r/2}}{\Delta^{12}}+\frac{q^{5m-r/4}}{\Delta^{8}}\big{)}  $$
and
$$\Delta \geq q^{\epsilon m}(q^{3m/2-r/2} + q^{5m/8-r/8}). $$
Thus by Lemma \ref{lem:K_delta}, either 
$$K \ll 1 $$
or 
$$K |\cF| \ll \frac{q^{3m+o(m)}}{\Delta^4}.  $$

In the first case we obtain 
$$E_{F,4}^{\sqrt}(\bm{1}_m) \leq q^{o(m)}\Delta^4K|\cF| \leq q^{o(m)}\Delta^4|\cD|^2 \leq q^{4m + o(m)} . $$

In the second case we would have 
\begin{align*}
    E_{F,4}^{\sqrt}(\bm{1}_m)
    &\leq q^{o(m)}\Delta^4K\frac{q^{3m + o(m})}{\Delta^4}\\
    &\leq q^{3m + o(m)}K\\
    &\leq q^{5m + o(m)}.
\end{align*}

Thus all together we have 
$$E_{F, 4}^\sqrt(\bm{1}_m) \leq q^{6m + \epsilon m + o(m)}\big{(}{q^{11m/2-r/2}}+{q^{3m-r/4}} + q^{5m/8-r/8}\big{)} + q^{5m + o(m)} $$
and the result follows, since $\epsilon$ is arbitrary.

 \section*{Acknowledgements}
 
The authors are grateful to Igor Shparlinski for many helpful comments and discussions throughout the preparation of this paper.
 
During the preparation of this work, C.B. was supported by an Australian Government Research Training Program (RTP) Scholarship and 
B.K. by the Australian Research Council (DE220100859). 

\bibliographystyle{plain} 

\bibliography{refs}

\appendix
\newpage
\section{Notation Guide}\label{notationguide}
\begin{center}
    \begin{longtable}{ p{3cm} p{8cm}  }
    \hline 
    \multicolumn{2}{l}{General Notation} \\
  $q$ & a prime power; if $q$ is required to be odd, then this is specified \\
  $\F_q$ & the finite field of order $q$ \\
   $\F_q[T]$ &  the ring of univariate polynomials with coefficients from $\F_q$\\
   $\F_q(T)$ & the field of fractions of $\F_q[T]$\\
   $\Ki$ & the field of Laurent series in $1/T$ over $\F_q$, or equivalently the completion of $\F_q(T)$ with respect to $|~\cdot~|$\\
    $F$     &     a polynomial in $\F_q[T]$ \\
    $r$     &      the degree of $F$     \\
       $|~\cdot~|$ & the absolute value on $\Ki$ \\
    $\|~\cdot~\|$ & the sup-norm on $\Ki^d$ induced by $|~\cdot~|$\\
    $e(~\cdot~)$ & the canonical additive character of  $\Ki$\\
    $\mu$ & a Haar measure on $\Ki$, normalized such that the unit ball has measure $1$ \\
            $\deg_Fx$ & for $x \in \F_q[T]$, the degree of the unique polynomial $x'$ satisfying $\deg x' < r$ and $x' \equiv x \Mod{F}$\\
    $\ov{x}$ & for $x \in \F_q[T]$, the multiplicative inverse of $x$ modulo $F$ (if the inverse taken to a different modulus, this is specified)\\
          \hline
      \multicolumn{2}{l}{Geometry of Numbers} \\
      $\cL$ & a lattice in $\Ki^d$\\
      $A_\cL$ & an element of $\textup{GL}(\Ki^d)$ such that $\cL = A_\cL \F_q[T]^d$\\
      $\det \cL$ & the determinant of $A_\cL$, which can be interpreted as the measure (with respect to $\mu$) of one of it's fundamental cells \\
      $\cL^*$ & the dual lattice to $\cL$ \\
      $\cB$ & a convex body in $\Ki^d$\\
      $U_\cB$ & an element of $\textup{GL}(\Ki^d)$ such that $\cB = U_\cB B_1^d$\\
      $\vol \cB$ & the determinant of $U_\cB$, which can be interpreted as measure of $\cB$ with respect to $\mu$\\ 
      $\cB^*$ & the dual body to $\cB$ \\
      $N_\cB$ & a norm function induced by $\cB$ \\
      $\sigma_1,...,\sigma_d$ & the successive minima of $\cL$ with respect to $\cB$ \\
            $\sigma_1^*,...,\sigma_d^*$ & the successive minima of $\cL^*$ with respect to $\cB^*$ \\
    $\langle ~, \rangle$ & the dot product on $\Ki^d$\\
     \hline
    \multicolumn{2}{l}{Other Sets and Weights} \\
    $B_m(s)$        &     the set $\{x \in \Ki^d : \|x-s\| < q^m\}$ \\ 
    $B_m$        &     the set $B_m(0)\subseteq \Ki$ \\
    $\cI_m(s)$ & the set $B_m(s) \cap \F_q[T]^d$\\
    $\cI_m$ & the set $B_m \cap \F_q[T]$\\
    $\cS$ & typically denotes some finite subset of $\Ki$, but most often $\cS \subseteq \F_q[T]$\\
    $|\cS|$ & the cardinality of $\cS$\\
    $\cS(x)$ & given some $x \in \Ki$, $\cS(x)$ is equal to $1$ if $x \in \cS$ and $0$ otherwise\\
     $\balpha$ & a sequence of complex weights, most often on $\F_q[T]$ such that $\balpha$ is periodic modulo $F$ and has non-zero support only on $\{\deg_Fx < m\}$\\
     $\bm{1}_m$ & the characteristic function on $\{\deg_F x < m\}$\\
\hline 
    \multicolumn{2}{l}{Main Counting Functions} \\
        $C_{F, \Phi}(\cS)$ & the number of solutions to 
    \begin{center}
       $\Phi(x_1,x_2) \equiv 0 \Mod{F}$ 
    \end{center}
    with $(x_1, x_2) \in \cS \subseteq \F_q[T]^2$ and $\Phi(x,y) \in \F_q[T][x,y]$ \\
     $E_{F, k}^\inv(\cS)$ & the number of solutions to 
    \begin{center}
       $\ov{x_1} + ... + \ov{x_k} \equiv \ov{x_{k+1}} + ... + \ov{x_{2k}} \Mod{F}$ 
    \end{center}
    with $x_i \in \cS \subseteq \F_q[T]$\\
 $E^{\sqrt}_{F, k}(\balpha)$ & the sum 
 \begin{center}
$\displaystyle \smashoperator[r]{\sum_{\substack{(x_1,...,x_{2k}) \in \cI_r^k \\  x_1+...+x_k \equiv x_{k+1}+...+x_{2k} (F)}}}\:\:\:\:\alpha({x_1^2})\ov{\alpha({x_2^2})}\cdots\alpha({x_{2k-1}^2})\ov{\alpha({x_{2k}^2})}$
     \end{center}
     with $\balpha$ as above. If $\balpha = \bm{1}_m$ then the sum is equal to the number of solutions to 
     \begin{center}
         $x_1 + ... + x_{k} \equiv x_{k+1} + ... + {x_{2k}} \Mod{F}$
     \end{center}
     with $x_i \in \cI_r$ and $\deg_F(x_i^2) < m$
     \\
     \multicolumn{2}{l}{Other Counting Functions} \\
     $I_{F, \lambda, k }(\cS)$ & the number of solutions to 
    \begin{center}
       $\ov{{x}_1} + ...+\ov{{x}_k} \equiv \lambda \Mod{F}$ 
    \end{center}
    with $x_i \in \cS \subseteq \F_q[T]$\\
    $J_{\lambda, k,s}(\cS)$ & the number of solutions to the system
    \begin{center}
       $x_1^j + ... + x_s^j = \lambda_j, ~ 1\leq j \leq k$ 
    \end{center}
    with $x_i \in \cS \subseteq \F_q[T]$ and where $\lambda = (\lambda_1,...,\lambda_k) \in \F_q[T]^k$\\
    $Q_{F, \lambda}(\balpha)$ & the sum 
    \begin{center}
        $\displaystyle\sum_{\substack{x_1, x_2 \in \cI_r^2 \\ x_1-x_2 \equiv \lambda(F)}}\balpha(x_1^2)\ov{\balpha(x_2^2)}$
    \end{center}
    with $\balpha$ as above\\
        $Q_{F, \lambda, m}$ & shorthand for $Q_{F, \lambda}(\bm{1}_m)$, or the number of solutions to 
        \begin{center}
            $x_1-x_2 \equiv \lambda \Mod{F}$
        \end{center}
        with $x_i \in \cI_r$ and $\deg_F(x_i^2) < m$\\
     \hline 
\multicolumn{2}{l}{Miscellaneous} \\
$(f * g)(x)$ & the convolution $\sum_{y}f(y)g(x-y)$\\
$f^{(k)}(x)$ & the convolution $(f^{(k-1)}*f)(x)$\\
$h(P)$ & the height of $P$; given a polynomial $P \in \F_q[T][x]$, the degree of its largest coefficient\\
$\textup{Res}(P,Q)$ & the resultant of $P, Q \in \F_q[T][x]$\\
 \end{longtable}
\end{center}

\end{document}